\documentclass[11pt]{amsart}
\usepackage{amsmath, amsthm, amssymb}
\usepackage[T1]{fontenc}
\usepackage{eucal}
\DeclareMathAlphabet{\mathcal}{OMS}{cmsy}{m}{n}
\usepackage[all]{xy}
\usepackage{mathtools}
\usepackage{enumitem}
\usepackage{tikz}
\usepackage{skak}
\usepackage{graphics}
\usepackage{array}

\usepackage{kbordermatrix} %new package

\newcommand{\D}{\mathcal{D}}
\newcommand{\R}{\mathcal{R}}

\newcommand{\E}{\mathcal{E}}
\newcommand{\A}{\mathcal{A}}
\renewcommand{\P}{\mathcal{P}}
\newcommand{\B}{\mathcal{B}}
\newcommand{\C}{\mathcal{C}}

\newcommand{\F}{\mathcal{F}}

\newcommand{\mcL}{\mathcal{L}}

\newcommand{\mcO}{\mathcal{O}}

\newcommand{\fl}{\preceq}

\newtheorem{prop}{Proposition}[section]
\newtheorem{thm}[prop]{Theorem}
\newtheorem{cor}[prop]{Corollary}
\newtheorem{lem}[prop]{Lemma}
\theoremstyle{definition}
\newtheorem{defn}[prop]{Definition}

\newtheorem{exmp}[prop]{Example}
\newtheorem*{assump}{Standing Assumptions}
\newtheorem{rem}[prop]{Remark}

\newlist{thmenum}{enumerate}{10}
\setlist[thmenum,1]{label=\textnormal{(\alph*)}}
\setlist[thmenum,2]{label=\textnormal{(\roman*)}}

\begin{document}

\title[Labelled Graphs and Morita Equivalence]{Labelled Graphs as Morita equivalence invariants for a class of inverse semigroups}

\author[Z. Duah]{Zachary Duah}
\address{Department of Mathematics\\
University of Michigan\\
530 Church Street\\
Ann Arbor, MI 48109}
\email{duah.zach@gmail.com}

\author[S. Du Preez]{Stian Du Preez}
\address{Department of Mathematics\\
Rice University\\
PO Box 1892\\
Houston, TX 77005-1892}
\email{stiandupreez6@gmail.com}

\author[D. Milan]{David Milan}
\address{Department of Mathematics\\
The University of Texas at Tyler\\
3900 University Boulevard\\
Tyler, TX 75799}
\email{dmilan@uttyler.edu}

\author[S. Ramamurthy]{Shreyas Ramamurthy}
\address{Department of Mathematics\\
University of California, Berkeley\\
970 Evans Hall \#3840\\
Berkeley, CA 94720-3840}
\email{shreyas.ramamurthy@berkeley.edu}

\author[L. Vega]{Lucas Vega}
\address{Department of Mathematics\\
The University of Texas at Tyler\\
3900 University Boulevard\\
Tyler, TX 75799}
\email{lvega2@patriots.uttyler.edu}

\thanks{The authors were supported by an NSF grant (DMS-2149921).}

\date{\today}

\begin{abstract} We investigate the use of labelled graphs as a Morita equivalence invariant for inverse semigroups. We construct a labelled graph from a combinatorial inverse semigroup $S$ with $0$ admitting a special set of idempotent $\D$-class representatives and show that $S$ is Morita equivalent to a labelled graph inverse semigroup. For the inverse hull $S$ of a Markov shift, we show that the labelled graph determines the Morita equivalence class of $S$ among all other inverse hulls of Markov shifts.
\end{abstract}

\maketitle

\section{Introduction}

In his influential work on the tight groupoid of an inverse semigroup, Exel \cite{ExelBig} summarized the construction of $C^*$-algebras from combinatorial objects with the following diagram:
\[\fbox{combinatorial object}\to \fbox{inverse semigroup}\to \fbox{groupoid}\to \fbox{$C^*$-algebra}\]
The standard example of such a construction starts with a directed graph $\Gamma$ and concludes with the graph $C^*$-algebra of $\Gamma$. 

In this paper, we are interested in reversing the first arrow. We want to recover combinatorial data from an inverse semigroup that is sufficient to characterize the inverse semigroup, at least up to Morita equivalence. This has potential applications in $C^*$-algebra theory, since for example one could develop transformations similar to the graph moves in S{\o}renson \cite{SorensenGraph} on a broader collection of combinatorial objects as a path towards ``geometric'' classification results for large classes of inverse semigroup $C^*$-algebras.

Inspired by the characterization of graph inverse semigroups in \cite{LawsonGraph}, it was shown in \cite{REU2022} that from an inverse semigroup $S$ one can construct a directed graph whose vertices are the nonzero $\mathcal{D}$-classes of $S$ and whose edges are defined using the natural partial order on the idempotents. This directed graph is a Morita equivalence invariant for inverse semigroups. It is a complete invariant for combinatorial inverse semigroups with $0$ satisfying two additional properties:
\begin{enumerate}
\item for $ e,f,g \in E(S)$ with $e,f \leq g$ and $ef \neq 0$, $e$ and $f$ are comparable.
\item for $0 \neq e \leq f$, there are finitely many idempotents between $e$ and $f$.
\end{enumerate}
Moreover, the inverse semigroups that satisfy the above properties are precisely the inverse semigroups that are Morita equivalent to graph inverse semigroups. 

If $S$ does not satisfy property (1), we say that $S$ \emph{contains a diamond}. If $S$ contains a diamond, then the directed graph associated with $S$ will be insufficient to characterize Morita equivalence---we need a more detailed invariant than a directed graph. We replace (1) with the property that there is a partial order $\fl$ on the $\D$-classes of $S$ making $S/\D$ into a meet semilattice and satisfying a number of other properties. We can then pass to a Morita equivalent inverse semigroup $S^\fl$ that admits a \emph{coherent} set $\C$ of idempotent $\D$-class representatives (see Definition \ref{def:coherent}). Given a coherent set $\C$, we construct a labelled graph and consider its inverse semigroup which is essentially the same as the one defined and studied by Boava, de Castro, and de L. Mortari in \cite{labelledgraphs}. We show in Theorem \ref{thm:labelledgraph} that any combinatorial inverse semigroup with $0$ having a coherent set $\C$ and finite intervals is Morita equivalent to the inverse semigroup of the labelled graph. This generalizes results of \cite{REU2022} to many inverse semigroups containing diamonds. We use this result to prove in Theorem \ref{thm:markovmorita} that the labelled directed graph associated to the inverse semigroup of a Markov shift is a complete invariant among inverse hulls of Markov shifts. Consequently, to determine Morita equivalence of two inverse hulls of Markov shifts, one only needs to examine a finite portion of each semilattice.

\section{Morita Equivalence}
In this section we give a number of definitions and useful results pertaining to Morita equivalence of inverse semigroups. For a more thorough treatment of the subject, see the thesis of Afara \cite{AfaraThesis}, \cite{SteinbergMorita}, or \cite{FunkLawsonSteinberg}. Among the many equivalent definitions of Morita equivalence, the one given in terms of category equivalence is usually the most convenient. The \emph{idempotent splitting} of an inverse semigroup $S$ is a category $C(S)$ with objects $E(S)$ and morphisms $\{(e,s,f): e,f \in E(S), s \in eSf\}$. Composition is given by
\[
	(e,s,f)(f,t,g) = (e, st, g).
\]
Two objects (idempotents) $e,f$ are isomorphic in $C(S)$ if and only if they are $\D$-related in $S$. We write $[s]$ for the $\D$-class of $s$ in $S$. Given inverse semigroups $S$ and $T$, $S$ is \emph{Morita equivalent} to $T$ if and only if the categories $C(S)$ and $C(T)$ are equivalent. Let $T$ be an inverse subsemigroup of an inverse semigroup $S$. Then $S$ is an \emph{enlargement} of $T$ if $STS = S$ and $TST = T$. We will make use of the fact that if $S$ is an enlargement of $T$, then $S$ is Morita equivalent to $T$.
 
\begin{lem}\label{lem:equivalence}
    Suppose that $S$ and $T$ are combinatorial inverse semigroups and $\F:C(S) \to C(T)$ is an equivalence functor. Then
    \begin{itemize}
        \item [(1)] the map $\sigma([e]) = [\F(e)]$ defines a bijection $\sigma : S/\D \to T/\D$, and
        \item [(2)] for $e \in E(S)$ and $f \in E(T)$ with $\F(e) \D f$, there exists an isomorphism $\tau_{e,f}:eSe \to fTf$ such that for all $g \leq e$, we have $\sigma([g]) = [\tau_{e,f}(g)]$.
%and ${\tau_{e,f}}_{|gSg} = \tau_{g,\tau_{e,f}(g)}$.
    \end{itemize}
\end{lem}

\begin{proof}
    For $e \in E(S)$, let $\sigma([e]) = [\F(e)]$. Since two objects in the idempotent splitting are isomorphic if and only if they are $\D$-related, it is well known that $\sigma$ is a bijection between $S / \D$ and $T / \D$.

    Fix $e \in E(S)$ and $f \in E(T)$ with $f \D \F(e)$. Let $\tau_{e,f,1}: eSe \to \F(e)T\F(e)$ be the isomorphism defined by the equivalence functor, that is $\tau_{e,f,1}(s) = t$ where $\F(e,s,e) = (\F(e),t,\F(e))$. Choose $x_{e,f} \in T$ such that $x_{e,f}^*x_{e,f} = \F(e)$ and $x_{e,f}x_{e,f}^* = f$. Define $\tau_{e,f,2} : \F(e)T\F(e) \to fTf$ by $\tau_{e,f,2}(t) = x_{e,f} t x_{e,f}^*$. This is known to be an isomorphism that preserves $\D$-classes. Then $\tau_{e,f} = \tau_{e,f,2} \circ \tau_{e,f,1}$ is an isomorphism from $eSe$ to $fTf$. Finally, let $g \leq e$. Choose $y$ such that $\F(e,g,g) = (\F(e),y,\F(g))$. Then $yy^* = \tau_{e,f,1}(g)$ and $y^*y = \F(g)$. Since $\tau_{e,f,2}$ preserves $\D$-classes
    \[
        \sigma([g]) = [\F(g)] = [\tau_{e,f,1}(g)] = [(\tau_{e,f,2} \circ \tau_{e,f,1})(g)] = [\tau_{e,f}(g)].
    \]
\end{proof}

A number of results in this paper assume we have chosen a \emph{complete set $\C$ of idempotent representatives} of the $\D$-classes of an inverse semigroup $S$. By this we mean that $\C$ is a set of idempotents such that each $\D$-class of $S$ has a unique representative in $\C$. We now give some useful facts related to this situation. 

\begin{prop}\label{prop:enlargement} Suppose $S$ is an inverse semigroup and that $\C \subseteq E(S)$ contains a representative of each $\mathcal{D}$-class of $S$. Then $S$ is Morita equivalent to the inverse semigroup $\C S \C$.
\end{prop}
\begin{proof} Let $T = \C S \C$. We will prove that $S$ is an enlargement of $T$. One can quickly show that $T$ is an inverse subsemigroup of $S$ for which $T = TST$. We prove that $S = S T S$. First, we clearly have $STS \subseteq S$. Let $s \in S$. Then there exists $x \in S$ such that $s^*s = x^* x$ and $x x^* \in \C$. Then $s = s x^* x = sx^* (xx^*) x \in STS$. So $S = STS$.
\end{proof}

Afara and Lawson \cite{AfaraLawson} showed how to construct all inverse semigroups Morita equivalent to a given inverse semigroup $S$ using McAlister functions. A map $p:I\times I \to S$ is called a \emph{McAlister function} if it satisfies the following properties:
\begin{enumerate}
\item[(MF1):] $p_{i,i}$ is an idempotent for all $i \in I$.
\item[(MF2):] $p_{i,i}p_{i,j}p_{j,j} = p_{i,j}$.
\item[(MF3):] $p_{i,j} = p_{j,i}^{-1}$.
\item[(MF4):] $p_{i,j}p_{j,k} \leq p_{i,k}$. 
\item[(MF5):] For each $e \in E(S)$ there exists $i \in I$ such that $e \leq p_{i,i}$.
\end{enumerate}

From a McAlister function one constructs the inverse semigroup $IM(S, I, p)$ consisting of equivalence classes $[i, s, j]$ under a relation $\gamma$ defined on triples $(i,s,j)$ where $ss^* \leq p_{i,i}$ and $s^*s \leq p_{j,j}$. The equivalence relation $\gamma$ is defined by $(i,s,j) \gamma (k, t, l)$ if and only if $s = p_{i,k} t p_{l,j}$ and $t = p_{k,i} s p_{j,l}$. Multiplication in $IM(S,I,p)$ is given by $[i,s,j][k,t,l] = [i, s p_{j,k} t, l]$. Then $IM(S,I,p)$ is Morita equivalent to $S$. Moreover, if $T$ is any inverse semigroup that is Morita equivalent to $S$, then $T$ is isomorphic to $IM(S,I,p)$ for some set $I$ and McAlister function $p$.

Suppose now that $S$ is an inverse semigroup with $0$ and let 
\[ 
    \C = \{ e_v : v \in S / \D\} 
\]
be a complete set of idempotent representatives of the $\D$-classes of $S$. We assume that $\C$ is \emph{dominating} in the following sense: for every idempotent $e \in E(S)$, there exists $v \in S/\D$ such that $e \leq e_v$. We also assume that there is a partial order $\fl$ on $S/\D$ making it into a meet semilattice and satisfying the following for all $u,v, w \in S/\D$:
\begin{enumerate}

\item if $u \wedge v \neq 0$, then $e_u e_v = e_{u \wedge v}$, and

\item if $u,v \fl w$, then $e_u e_v = e_{u \wedge v}$.

\end{enumerate}

Let $I = S / \D$ and define $p: I \times I \to S$ by $p_{u,v} = e_{u \wedge v}$. Then $p$ satisfies (MF1)-(MF5). To see (MF4), fix $u,v,w \in I$ and suppose that $0 \neq p_{u,v}p_{v,w}$. Since $u \wedge v, v \wedge w \fl v$ we have
\[
    0 \neq p_{u,v}p_{v,w} = e_{u \wedge v}e_{v \wedge w} = e_{u \wedge v \wedge w}.
\]
In particular, $0 \neq u \wedge v \wedge w \fl u \wedge w$. Thus
\[
    p_{u,v}p_{v,w} = e_u e_v e_w \leq e_u e_w = p_{u,w}.
\]

Let $S^{\fl} = IM(S,I,p)$. From the work of Afara and Lawson, we have that $S$ is Morita equivalent to $S^{\fl}$.

We now describe the idempotents and the partial order on $S^{\fl}$. One can quickly check that $(u,s,v) \gamma (w,t,z) $ if and only if
\begin{enumerate}

\item $s = t$, $u \wedge w \neq 0$, and $v \wedge z \neq 0$, or

\item $s = t = 0$.

\end{enumerate}

So $S^{\fl}$ has a zero with $[u,s,v] = 0$ if and only if $s = 0$. We also have
\begin{align*}
    E(S^{\fl}) &= \{[u,ss^*,u]: u \in S/\D, ss^* \leq p_{u,u} \} \\
    & = \{[u,e,u]: u \in S/\D, e \in e_u^\downarrow\}.
\end{align*}

We leave the proof of the following proposition to the reader.

\begin{prop}\label{prop:flD} Let $S^{\fl}= IM(S,I,p)$ as described above and suppose that $[a,f,a]$ and $[b,g,b]$ are idempotents in $S^{\fl}$. Then $[a,f,a] \D [b,g,b]$ if and only if $f \D g$.
\end{prop}

Finally, we describe the natural partial order.

\begin{prop}\label{prop:flineq} Let $S^{\fl}= IM(S,I,p)$ as described above and suppose $[u,s,v] \neq 0$. Then $[u, s, v] \leq [w, t, z]$ in $S^{\fl}$ if and only if $s \leq t, u \wedge w \neq 0,$ and $v \wedge z \neq 0$.
\end{prop}
\begin{proof}
Suppose $[u, s, v] \leq [w, t, z]$. Then 
\[
[u,s,v] = [w, t, z] [v, s^*, u] [u, s, v] = [w, t p_{z,v} s^* e_u s, v].
\]
So $ s = t p_{z,v} s^* e_u s$ and $u \wedge w \neq 0$. Also, since $[u,s,v] \neq 0$ we have $z \wedge v \neq 0$ and 
\[
    s = t p_{z,v} s^* e_u s = t e_z e_v s^* e_u s = t s^* s.
\]
The converse is similar.

\end{proof}

\section{Labelled Graph Inverse Semigroups}

In this section we consider inverse semigroups associated with labelled graphs. The semigroups we consider are essentially the same as the ones defined in Boava, de Castro, and de L. Mortari \cite{labelledgraphs}. One difference is that we use a different convention for the composition of paths. More significantly, we remove the requirement that a certain collection of vertices is closed under unions. This does not change the definition of the product in the inverse semigroup associated with a labelled graph space.

\begin{defn}Given a set $\A$, called an \textit{alphabet}, containing \textit{letters}, a \textit{labelled graph} $(\E,\mcL)$, over $\A$ consists of a directed graph $\E= (\E^0,\E^1, r,s)$ and a map $\mcL:\E^1  \to A$, called a \textit{labelling map}.
\end{defn}

Let $\A^*$ be the set of all finite words over $\A$ including the empty word, $\omega$, and extend the map $\mcL$ to a map $\mcL:\E^n \to \A^*$ defined by $\mcL(\lambda) = \mcL(\lambda_1)\ldots\mcL(\lambda_n)$ for $\lambda = \lambda_1\ldots \lambda_n \in \E^1$.

The elements of $\mcL^n := \mcL(\E^n)$ are called \textit{labelled paths of length $n$}. If $\alpha$ is a labelled path, a path on the graph $\lambda$ s.t. $\mcL(\lambda) = \alpha$ is called a \textit{representative} of $\alpha$, and conversely $\alpha$ is the \textit{label} of $\lambda$. Since any two representatives of $\alpha$ have the same length, one can define the \textit{length} of $\alpha$, denoted by $|\alpha|$, as the length of any one of its representatives. We also consider $\omega$ as a labelled path, with $|\omega| = 0$. The set $\mcL^{\geq 1} = \cup_{n \geq 1}\mcL^n$ is the set of all labelled paths of positive finite length. We also define $\mcL^* = \{\omega\} \cap \mcL^{\geq 1}$.

Given two labelled paths $\alpha,\beta$, we say that $\beta$ \textit{is a sublabel of} $\alpha$ if $\alpha = \beta\alpha'$ for some labelled path $\alpha'$. We say that $\alpha$ and $\beta$ are \textit{comparable} if $\alpha$ is a sublabel of $\beta$ or $\beta$ is a sublabel of $\alpha$.

For $\alpha \in \mcL^*$ and $A \in \P(\E^0)$, the \textit{relative source of $\alpha$ with respect to $A$}, denoted by $s(A,\alpha)$, is the set
\[s(A,\alpha) = \{s(\lambda) : \lambda \in \E^*, \mcL(\lambda) = \alpha, r(\lambda) \in A\}\]
if $\alpha \in \mcL^{\geq 1}$ and $s(A,\alpha) = A$ if $\alpha = \omega$. The \textit{source} of $\alpha$, denoted by $s(\alpha)$, is the set 
\[s(\alpha) = s(\E^0,\alpha)\]

For $\alpha \in \mcL^{\geq 1}$ we also define the \textit{range} of $\alpha$ as the set 
\[r(\alpha) = \{r(\alpha) \in \E^0 : \mcL(\lambda) = \alpha\}\]
Observe that, in particular, $s(\omega) = \E^0$, and if $\alpha \in \mcL^{\geq 1}$ then $s(\alpha) =\{s(\lambda) \in \E^0 : \mcL(\lambda) = \alpha\}$. The definitions above allow us to extend the $s,r$ maps to $s,r : \mcL^{\geq 1} \to \P(\E^0)$. Also, if $\alpha,\beta \in \mcL^*$ are such that $\beta\alpha \in \mcL^*$ then $s(s(A,\beta),\alpha) = s(A, \beta\alpha)$. Finally, for $A,B \in \P(\E^0)$ and $\alpha \in \mcL^*$, it holds that $s(A \cup B, \alpha) = s(A,\alpha) \cup s(B,\alpha)$.

Let $(\E,\mcL)$ be a labelled graph and $\B \subseteq \P(\E^0)$. We say that $\B$ is \textit{closed under relative sources} if $s(A,\alpha) \in \B$ for all $A \in \B$ and all $\alpha \in \mcL^*$.  If additionally $\B$ is closed under finite intersections and contains all $s(\alpha)$ for $\alpha \in \mcL^{\geq 1}$, we say that $(\E,\mcL,\B)$ is a \textit{labelled graph space}.

\begin{rem} This is where our definition differs from the one in \cite{labelledgraphs}. We do not assume that $\B$ is closed under unions in a labelled graph space.
\end{rem}

We say that a labelled graph space $(\E,\mcL,\B)$ is \textit{weakly right resolving} if for all $A,B \in \B$ and for all $\alpha \in \mcL^{\geq 1}$ we have $s(A \cap B, \alpha) = s(A,\alpha) \cap s(B,\alpha)$.

Observe that in a weakly right resolving labelled graph space, if $A,B \in \B$ are disjoint then for all $\alpha \in \mcL^*, s(A,\alpha)$ and $s(B,\alpha)$ are disjoint. For a given $\alpha \in \mcL^*$, let $\B_\alpha := \B \cap \P(s(\alpha))$.

We can now define the inverse semigroup of $(\E,\mcL,\B)$. Many of the proofs are omitted as they are nearly identical to the proofs in \cite{labelledgraphs}. Let $T$ be the set of all triples $(\alpha, A, \beta) \in \mcL^* \times \B \times \mcL^*$ for which $A \in \B_\alpha \cap \B_\beta$, together with an extra element $z$. A binary operation on $T$ is defined as follows: for all $t \in T$ define $zt = tz = z$, and for elements $(\alpha, A, \beta), (\gamma, B, \delta)$ in $\mcL^*\times \B\times \mcL^*$ put
\[(\alpha, A, \beta) (\gamma, B, \delta) = \begin{cases}
(\alpha\gamma', s(A,\gamma')\cap B , \delta) & \text{if }\gamma = \beta\gamma'\\
(\alpha,A \cap s(B, \beta') , \delta\beta')& \text{if }\beta = \gamma\beta'\\
z& \text{o.w.}
\end{cases}\]

Thus, nontrivial products occur only if $\beta$ and $\gamma$ are comparable. This product is well-defined: in the first case $A \in \B_\alpha$ implies $s(A,\gamma') \in \B_{\alpha\gamma'}$ which together with $B \in \B_\delta$ gives $s(A,\gamma') \cap B \in \B_{\alpha\gamma'} \cap \B_\delta$. In the second case $B \in \B_\delta$ implies $s(B,\beta') \in \B_{\delta\beta'}$ which together with $A \in \B_\alpha$ gives $A \cap s(B, \beta') \in \B_{\delta\beta'} \cap \B_\alpha$.

If $\beta = \gamma$ then $\gamma' = \beta' = \omega$ and, therefore, $s(A,\gamma') = A, s(B,\beta') = B$, and
\[(\alpha, A, \beta) (\beta, B, \delta) = (\alpha, A \cap B, \delta).\]

The proof of the following proposition is nearly identical to the proof of \cite[Proposition 3.2]{labelledgraphs}.
\begin{prop}
    Suppose that the labelled graph space $(\E,\mcL,\B)$ is weakly right resolving. Then the set $T$ together with the operation above is a semigroup with zero.
\end{prop}

Next we obtain a set $S$ from $T$ by collapsing $z$ and all elements in $T$ of the form $(\alpha, \emptyset,\beta)$ to a single element, $0$, and leaving all other elements in $T$ as they are. $S$ inherits the associative operation of $T$, meaning $S$ is a semigroup with $0$. We need to show that operation is well-defined with the zero class. Let $s = (\alpha,\emptyset,\beta)$ and $t = (\gamma,B,\delta)$.

First examine the product $st$ (which should be $0$ since $s = 0$). If $\beta$ and $\gamma$ are not comparable, then $st = z = 0$. If $\gamma  = \beta\gamma'$, then $st = (\alpha\gamma', s(\emptyset,\gamma')\cap B , \delta) = (\alpha\gamma', \emptyset\cap B , \delta) = (\alpha\gamma', \emptyset , \delta) = 0$. And if $\beta = \gamma\beta'$, then $st = (\alpha,\emptyset \cap s(B, \beta') , \delta\beta') =(\alpha,\emptyset , \delta\beta')  = 0$.

And the product $ts$. If $\delta$ and $\alpha$ are not comparable, then $ts = z = 0$. If $\delta = \alpha\delta'$, then $ts = (\gamma,B \cap s(\emptyset, \delta'),\beta\delta') = (\gamma,B \cap \emptyset,\beta\delta') = (\gamma,\emptyset,\beta\delta') = 0$. And if $\alpha = \delta\alpha'$, then $ts = (\gamma\alpha', s(B,\alpha' )\cap \emptyset ,\beta) = (\gamma\alpha', \emptyset ,\beta) = 0$. Thus the operation on $S$ is well-defined with the zero class. See \cite[Proposition 3.4]{labelledgraphs} for the proof of the following proposition.
\begin{prop}
    Suppose the labelled graph space $(\E,\mcL,\B)$ is weakly right resolving. Then $S$ is an inverse semigroup with $0$.
\end{prop}

Finally, we state without proof the following proposition that describes the basic structure of the inverse semigroup of a labelled graph.
\begin{prop} Let $S$ be the inverse semigroup associated with the labelled graph space $(\E, \mcL, \B)$.
\begin{enumerate}

\item $E(S) = \{ (\alpha, A, \alpha): \alpha \in \mcL^*, A \in B_{\alpha}\} \cup \{0\}$.
\item $(\alpha, A, \beta) \leq (\gamma, B, \delta)$ if and only if there exists $\mu \in \mcL^*$ such that $\beta = \delta \mu, \alpha = \gamma \mu,$ and $A \subseteq s(B,\mu)$.
\item $(\alpha, A, \beta) \R (\gamma, B, \delta)$ if and only if $\alpha = \gamma$ and $A = B$.
\item $(\alpha, A, \beta) \mcL (\gamma, B, \delta)$ if and only if $\beta = \delta$ and $A = B$.
\item $S$ is combinatorial and $0$-E-unitary.

\end{enumerate}
\end{prop}

\section{Strongly Right Resolving Labelled Graphs}
 \begin{defn} Let $(\E,\mcL)$ be a labelled graph. We say $(\E,\mcL)$ is \textit{strongly right resolving} if for all $e,f \in \E^1$, $\mcL(e) = \mcL(f)$ implies $r(e) = r(f)$.
\end{defn}

\begin{prop} Suppose $\E = (\E^0,\E^1,r,s)$ is a directed graph, $\mcL : \E^1 \to A$ is a labelling map, and $(\E,\mcL)$ is strongly right resolving. Further suppose that $\B \subseteq \P(\E^0)$ contains the empty set, is closed under finite intersections, and contain all $s(\alpha)$ for $\alpha \in \mcL^{\geq 1}$, then $(\E,\mcL,\B)$ is a weakly right resolving labelled graph space.
\end{prop}
\begin{proof}
First, we show that $\B$ is closed under relative sources. Let $A \in \B, \alpha \in \mcL^*$. If $\alpha = \omega$, then $s(A,\omega) = A \in \B$, so assume otherwise. Then since $(\E, \mcL)$ is strongly right resolving, the range of all representatives of $\alpha$ must be the same, let $a \in \E^0$ be this range. If $a \notin A$, then clearly $s(A,\alpha) = \emptyset$. If $a \in A$, then $s(A,\alpha) = s(\alpha) \in A$, since the ranges of all representatives of $\alpha$ are in $A$. Thus $(\E,\mcL,\B)$ is a labelled graph space. 

Next we show that $(\E,\mcL,\B)$ is weakly right resolving. Let $A,B \in \B, \alpha \in \mcL^*$. If $\alpha = \omega$, then clearly $s(A\cap B, \omega) = A\cap B = s(A,\omega)\cap s(B,\omega)$, so suppose otherwise. Again, let $a \in \E^0$ be the range of all the representatives of $\alpha$. If $a \in A\cap B$, then $s(A \cap B,\alpha) = s(\alpha) = s(\alpha) \cap s(\alpha) = s(A,\alpha) \cap s(B,\alpha)$ and we are done. Suppose $a \notin A \cap B$. Without loss of generality, suppose $a \notin A$, then $s(A \cap B,\alpha) = \emptyset = \emptyset \cap s(B,\alpha) =  s(A,\alpha) \cap s(B,\alpha)$. Thus we are done.
\end{proof}

If $S$ is the labelled graph inverse semigroup of the labelled graph space $(\E,\mcL,\B)$ where $(\E,\mcL)$ is strongly right resolving, then $r(\alpha)$ is a singleton for all $\alpha \in \mcL^{\geq 1}$. We will use $r(\alpha)$ to refer to the element and not the set containing one element moving forward.

\begin{prop}
    Let $(\E,\mcL,\B)$ be a labelled graph space where $(\E,\mcL)$ is strongly right resolving. For $a \in \mcL^{1}$ and $A \in \B$,
    \[s(A,a) = \begin{cases}
        s(a) & \text{if }r(a) \in A,\\
        \emptyset & \text{if }r(a) \notin A.
    \end{cases}\]
\end{prop}
\begin{proof}
    Suppose $r(\alpha) \in A$, then
    \[s(A,a) = \{s(e) : e \in \E^*, \mcL(e) = a, r(e) \in A\}\]
    Since $(\E, \mcL)$ is strongly right resolving, $r(e) = r(a) \in A$ and 
    \[s(A,a) = \{s(e) : e \in \E^*, \mcL(e) = a\} = s(a).\]
    On the other hand, if $r(a) \notin A$, then no representative of $a$ has a range in $A$, so $s(A,a) = \emptyset$.
\end{proof}

\begin{prop}
     Let $(\E,\mcL,\B)$ be a labelled graph space where $(\E,\mcL)$ is strongly right resolving. If $\alpha = a_1 \ldots a_n \in \mcL^{\geq 1}$, then $r(a_{i+1}) \in s(a_{i})$ for $i = 1,\ldots, n-1$.
\end{prop}
\begin{proof}
    Since $\alpha \in \mcL^{\geq 1}$, there exists a representative $\lambda \in \E^*$, so $s(\alpha) \not = \emptyset$. Let $i =1,\ldots, n-1$ be arbitrary, then
\begin{align*}
    s(\alpha) &= s(a_1\ldots a_n) \\
            &= s(s(a_1\ldots a_{i-1}), a_{i},\ldots a_n) \\
            &= s(s(s(a_1\ldots a_{i-1}), a_{i}), a_{i+1}\ldots a_n) \\
            &= s(s(s(s(a_1\ldots a_{i-1}), a_{i}), a_{i+1}), a_{i+2}\ldots a_n)
\end{align*}
    Thus, if $r(a_{i+1}) \notin s(a_{i})$, then $s(\alpha) = \emptyset$ which we know is false, so $r(a_{i+1}) \in s(a_{i})$.
\end{proof}

\begin{prop}\label{prop:relsource} Let $S$ be a labelled graph inverse semigroup coming from the labelled graph space $(\E,\mcL,\B)$ where $(\E,\mcL)$ is strongly right resolving. For $\alpha \in \mcL^{1}, A \in \B$,
    \[s(A,\alpha) = \begin{cases}
        s(\alpha) & \text{if }r(\alpha) \in A\\
        \emptyset & \text{if }r(\alpha) \notin A
    \end{cases}\]
    and moreover, if $\alpha = a_1\ldots a_n$, then
    \[s(A,\alpha) = \begin{cases}
        s(a_n) & \text{if }r(a_1) \in A\\
        \emptyset & \text{if }r(a_1) \notin A
    \end{cases}\]
\end{prop}
\begin{proof}
    This follows as a corollary from the previous two propositions.
\end{proof}

\begin{prop}
    Let $S$ be a labelled graph inverse semigroup coming from the labelled graph space $(\E,\mcL,\B)$ where $(\E,\mcL)$ is strongly right resolving. Then for $(\alpha,A,\beta),(\gamma, B, \delta)\in S $ we have $(\alpha,A,\beta)(\gamma, B, \delta)=$
	\[
		\begin{cases}
        (\alpha\gamma', s(\gamma')\cap B , \delta) & \text{if }\gamma = \beta\gamma', \gamma' \not =\omega, r(\gamma') \in A, \text{and }s(\gamma')\cap B \not= \emptyset,\\
        (\alpha,A \cap s(\beta') , \delta\beta')& \text{if }\beta = \gamma\beta', \beta' \not = \omega, r(\beta') \in B, \text{and }A \cap s(\beta') \not= \emptyset,\\
        (\alpha, A\cap B, \delta)&  \text{if }\gamma = \beta \text{ and } A \cap B \not= \emptyset, \text{and}\\
        0& \text{otherwise}
    \end{cases}\]
\end{prop}
\begin{proof}
    This follows from the previous proposition and the definition of multiplication.
\end{proof}

\section{The Labelled Graph of an Inverse Semigroup}

In this section we construct a labelled graph from an inverse semigroup satisfying some key properties. Ultimately we seek to generalize the construction in \cite{REU2022} of a directed graph from a combinatorial inverse semigroup $S$ with $0$.

For idempotents $e,f$ we write $e \ll f$ to denote that $e$ is \emph{immediately below} $f$. That is, $e < f$ and there does not exist $g$ such that $e < g < f$. We say that an inverse semigroup \emph{has finite intervals} if for idempotents $e,f$ with $0 \neq e \leq f$ there exists idempotents $h_i$ for $1 \leq i \leq n$ such that:
\[
	0 \neq e = h_1 \ll h_2 \ll \dots \ll h_n = f
\]

We will use a set of idempotent $\D$-class representatives to construct our labelled graph. Recall that we denote the $\mathcal{D}$-class of $s$ by $[s]$.

\begin{defn}\label{def:coherent} A complete set $\C = \{e_v : v \in S/\D\}$ of idempotent representatives of the $\D$-classes of $S$ is a \emph{coherent set for $S$} if 
\begin{enumerate}

\item $e_u e_v \in \C$ for all $u,v \in \C$, 

\item $0 \neq e_u \leq f \leq e_v$ implies $f \in \C$ for all idempotents $f$, and

\item if $f,g \ll e_v$ with $f$ and $g$ incomparable, then $fg \in \C$.
\end{enumerate}
\end{defn}
Note that we will frequently make use of the observation that $f \in \C$ if and only if $f = e_{[f]}$. 

\begin{assump} Throughout this section we assume that $S$ is a combinatorial inverse semigroup with $0$ having finite intervals and admitting a coherent set $\C$ of idempotent representatives of the $\D$-classes of $S$.
\end{assump}

We will prove that any inverse semigroup $S$ with the above properties is Morita equivalent to a labelled graph inverse semigroup. Later, this result is used to show that the labelled graph associated to the inverse hull of a Markov shift is a complete Morita equivalence invariant among all inverse hulls of Markov shifts.

\begin{rem} Given a coherent set $\C = \{e_v: v \in S/\D\}$ for $S$, there is a partial order $\fl$ on $S/\D$ defined by $u \fl v$ if and only if $e_u \leq e_v$. By Definition \ref{def:coherent} (1), $(S/\D, \fl)$ is a meet semilattice. We write $u \wedge v$ for the meet of $u,v \in S/\D$. 
\end{rem}

Define the labelled graph $(\E,\mcL)$ associated with $\C$ as follows. First, the vertex set is  
\[\E^0 := S/\D-\{0\},\]
and for each $v \in \E^0$, we define $B_v := \{w \in \E^0: w \fl v\}$. Let
\[ \B := \{B_v : v \in \E^0\} \cup \{\emptyset\}.
\]

Fix a vertex $v \in \E^0$ and a nonzero idempotent $f$ with $f \ll e_v$ and $f \neq e_{[f]}$. Define an edge $x_{v,f,u}$ for all $u \in B_{[f]}$ with 
\[
r(x_{v,f,u}) = v \text{ and } s(x_{v,f,u}) = u.
\] 
Also assign the label $\mcL(x_{v,f,u}) = (v,f)$. We now have a directed graph $\E = (\E^0,\E^1,r,s)$ and a labelling $\mcL$.

% \begin{exmp} **It would be nice to put an example here. I'm thinking a simple Markov shift with a diamond. So we should probably define Markov shifts before this section.**
% \end{exmp}

\begin{prop} The triple $(\E, \mcL, \B)$ is a labelled graph space. Moreover the labelled graph $(\E, \mcL)$ is strongly right reductive.
\end{prop}

\begin{proof}
We first show that $\B$ is closed under finite intersections. For $u,v \in \E^0$, if $B_u \cap B_v = \emptyset$, then we are done since $ \emptyset \in \B$. Otherwise, $e_u e_v \neq 0$ (so $u \wedge v \in \E^0$) and we claim that $B_u \cap B_v = B_{u \wedge v}$. Let $w \in B_u \cap B_v$. Then $0 \neq e_w \leq e_u, e_v$. Thus $e_w \leq e_u e_v$ and $w \fl u \wedge v$. So $w \in B_{u \wedge v}$. Conversely, if $w \in B_{u \wedge v}$ then $w \fl u$ and $w \fl v$. So $w \in B_u \cap B_v$. It follows by induction that $\B$ is closed under finite intersections. 

Next we show that $(\E, \mcL)$ is strongly right reductive. Suppose that $x_{v_1, f, u_1}$ and $x_{v_2, f, u_2}$ are edges with the same label. Then $(v_1,f) = (v_2, f)$. In particular, $r(x_{v_1, f, u_1}) = v_1 = v_2 = r(x_{v_2, f, u_2})$.

Finally, we show that $\B$ is closed under relative sources. By Proposition \ref{prop:relsource}, it suffices to prove $s(a) \in \B$ for each $a = (v,f)$ in $\mcL^{1}$. If $s(a) = \emptyset$, we are done. Otherwise, notice that $u \in s(a)$ if and only if there exists an edge $(v,f,u)$ where $u \in B_{[f]}$. Thus $s(a) = B_{[f]} \in \B$.

\end{proof}

\subsection{$S$ contains the labelled graph inverse semigroup}
Let $S_\C$ be the labelled graph inverse semigroup of $(\E, \mcL, \B)$. We now show that $S_\C$ embeds in $S$.

For each $(v,f) \in \mcL^1$, we know that $f \not = e_{[f]}$. Let $s_{v,f} \in S$ be the unique element such that $s_{v,f}^*s_{v,f} = e_{[f]}$ and $s_{v,f}s_{v,f}^* = f$. For $\alpha \in \mcL^{\geq 1}$ with $\alpha = x_1\ldots x_n$ where $x_i = (v_i,f_i)$ for $i = 1,\ldots, n$, define 
\[s_\alpha := s_{x_1}\ldots s_{x_n} \text{ and } e_{\alpha} := s_\alpha s_\alpha^*.\]
Also, for $B_v \in \B$, let $e_{B_v} = e_v$. In this way, we have defined $e_{s(\alpha)}$ for each $\alpha \in \mcL^{\geq 1}$.

\begin{prop}
For each $\alpha \in \mcL^{\geq1}$, we have $s_{\alpha}^*s_{\alpha} = e_{s(\alpha)}$.
\end{prop}
\begin{proof}
    Let $\alpha = x_1\ldots x_n$ be a label where $x_i = (v_i,f_i)$ for $i = 1,\ldots, n$. Since $(\E, \mcL)$ is strongly right reductive, $v_{i+1} = r(x_{i+1}) \in s(x_{i})$ for $i = 1,\ldots, n - 1$. Thus, $v_{i+1} \in B_{[f_i]}$, and in particular $v_{i+1} \preceq [f_i]$. Thus $e_{v_{i+1}} \le e_{[f_i]} = s_{x_i}^*s_{x_i}$. Recall also that $s_{x_i}s_{x_i}^* = f_i \le e_{v_i}$ for $i = 1,\ldots, n$. Thus
    \begin{align*}
        s_\alpha^*s_\alpha &= s^*_{x_n}\ldots s_{x_3}^*s_{x_2}^*(s_{x_1}^*s_{x_1})(s_{x_2})s_{x_3}\ldots s_{x_n}\\
        & = s^*_{x_n}\ldots s_{x_3}^*s_{x_2}^*(e_{[f_1]}e_{v_2})s_{x_2}s_{x_3}\ldots s_{x_n}\\
        & = s^*_{x_n}\ldots s_{x_3}^*s_{x_2}^*(e_{v_2}s_{x_2})s_{x_3}\ldots s_{x_n}\\
        & = s^*_{x_n}\ldots s_{x_3}^*(s_{x_2}^*s_{x_2})(s_{x_3})\ldots s_{x_n}\\
        & \cdots\\
        & = s^*_{x_n}s_{x_n} = e_{[f_n]}
    \end{align*}
    Since $(\E, \mcL)$ is strongly right reductive, $s(\alpha) = s(x_n)$. Notice that the representatives of $x_n$ are exactly the set of $x_{v_n,f_n,u}$ where $u \in B_{[f_n]}$, and that the source of each such representative is $u$. Thus $s(x_n) = B_{[f_n]}$. So 
    \[s_\alpha^*s_\alpha = e_{[f_n]} = e_{s(x_n)} = e_{s(\alpha)}\]
    as desired.
\end{proof}

\begin{prop}\label{prop:threeproperties} Given the labelled graph space $(\E, \B, \mcL)$ defined above, we have the following properties:
\begin{enumerate}

\item for $x,y \in \mcL^1$, $s_x^*s_y \not = 0$ if and only if $x = y$,

\item for $x,y \in \mcL^1$, $e_x = e_y$ if and only if $x = y$, and

\item for $\alpha, \beta \in \mcL^{\geq 1}$, 
	\[ e_{s(\alpha)} s_{\beta} = \begin{cases} s_{\beta} & \text{ if $r(\beta) \in s(\alpha)$} \\
																			0 & \text{otherwise}. \end{cases}
	\]
\end{enumerate}
\end{prop}
\begin{proof}
For (1), let $x = (v,f), y = (u,g)$. If $x = y$, then $s_x^* s_x = e_{[f]} \neq 0$ by definition. Conversely, suppose that $s_x^*s_y \not = 0$. Then $(s_x^*s_xs_x^*)(s_ys_y^*s_y) \not = 0$, so $(s_xs_x^*)(s_ys_y^*)\not = 0$. Thus 
\[ 0 \not = (s_xs_x^*)(s_ys_y^*) = fg \le e_ve_u. 
\]

It follows that $e_{u \wedge v} = e_u e_v \neq 0$. We claim that $u = v = u \wedge v$. For contradiction, suppose that $e_{u \wedge v} < e_v$. Since $S$ has finite intervals, there exists $h$ such that $0 \neq e_{u \wedge v} \leq h \ll e_v$. Then $h \in C$ and hence $h \neq f$. Since $h \ll e_v$ and $f \ll e_v$ we know that $h$ and $f$ are incomparable. Thus $hf = 0$. We have $0 \neq fg \leq e_{u \wedge v} f = e_{u \wedge v} hf = 0$, a contradiction. Thus $v = u \wedge v$ and similarly $u = u \wedge v$.

We have $f \ll e_u$ and $g \ll e_u$. Since $fg \neq 0$ and $fg \not \in \C$, we conclude that $f$ and $g$ are comparable. Say without loss of generality that $f \leq g \leq e_u$. Then since $f \ll e_u$ it follows that $f = g$ or $g = e_u$. Since $g \not \in \C$, we must have that $f = g$. Thus $x = (v,f) = (u,g) = y$.

For (2), suppose that $e_x = e_y$, where $x = (v,f), y = (u,g)$. Then $s_x s_x^* s_y s_y^* = e_x e_y = e_x \neq 0$. So $s_x^* s_y \neq 0$ and hence $x = y$. The converse is clear.

To prove (3), we first show for $B_u \in \E^0$ and $x = (v,f) \in \mcL^1$, that 
    \[e_{B_u}s_x = \begin{cases}
        s_x & \text{if } r(x) \in B_u\\
        0 & \text{otherwise}
    \end{cases}\]
    Suppose $v = r(x) \in B_u$, i.e., $v \preceq u$. Thus $f \ll e_v \le e_u$, so 
\[
e_{B_u}s_x = e_u (s_xs_x^*)s_x = e_u f s_x = f s_x = s_xs_x^*s_x = s_x.
\]
 Conversely, suppose $v = r(x) \notin B_u$. Then $e_v \not\le e_u$. Suppose for the sake of obtaining a contradiction that $fe_u = fe_ve_u$ is nonzero. Since $e_v \not\le e_u$, it must be that $e_u e_v < e_v$. We consider three cases, each resulting in a contradiction. If $f$ and $e_u e_v$ are incomparable, then by Definition \ref{def:coherent} (3), $f e_u e_v \in \C$. Since $0 \neq f e_u e_v \leq f \ll e_v$, $f \in C$ by Definition \ref{def:coherent} (2). This is a contradiction. Similarly, if $e_u e_v \leq f$ then $f \in \C$, a contradiction. Finally, suppose that $f \leq e_u e_v$. Since $f \ll e_v$ and $f \leq e_u e_v < e_v$, we conclude that $f = e_u e_v$. Thus $f \in \C$, another contradiction.

The general case follows from the fact that if $s(\alpha) = B_u$ for some $u \in E_S^0$, and $\beta = x_1\ldots x_n$ where $x_1,\ldots, x_n \in \mcL^1$ and $r(\beta) = r(x_1)$. Then
\begin{align*}
e_{s(\alpha)}s_{\beta} &= e_{B_u}s_{x_1}\ldots s_{x_n} \\
					   &= \begin{cases}
        s_{x_1}\ldots s_{x_n} & \text{if }r(x_1)  \in B_u\\
        0s_{x_2}\ldots s_{x_n} & \text{otherwise}
    \end{cases} \\
    				  &= \begin{cases}
        s_\beta & \text{if } r(\beta) \in s(\alpha)\\
        0 & \text{otherwise}
    \end{cases}
\end{align*}

\end{proof}

\begin{cor}
For any word $\alpha \in \mcL^{\geq 1}$, we have $s_\alpha\not = 0$ and $e_{\alpha} \not = 0$.
\end{cor}

\begin{proof}
    Clearly $s_x \not = 0$ for all $x \in \mcL^1$ by construction. Let $\alpha = \alpha'x$ where $x \in \mcL^1$ and $\alpha' \in \mcL^*$. If $\alpha' = \omega$, then $s_\alpha = s_x \not = 0$. Otherwise, $e_{s(\alpha')}s_x = s_x \not =0$ since $r(x) \in s(\alpha')$. But then $0 \not = e_{s(\alpha')}s_x = s_{\alpha'}^*s_{\alpha'}s_x$. So 
    \[ s_{\alpha} = s_{\alpha'}s_x = s_{\alpha} \neq 0.\] 
Also, if $e_\alpha = 0$, then $s_\alpha s_\alpha^* = e_\alpha = 0$, so $s_\alpha = 0$. We conclude that $e_\alpha \neq 0$
\end{proof}

\begin{prop}\label{prop:pathprods} For $\alpha,\beta \in \mcL^{\geq 1}$ we have
    \[s_\alpha^*s_\beta = \begin{cases}
        e_{s(\alpha)} & \text{if }\alpha = \beta\\
        s_{\beta'} & \text{if }\beta = \alpha\beta', \beta' \not = \omega \\
        s^*_{\alpha'}& \text{if }\alpha = \beta\alpha', \alpha' \not = \omega\\
        0 & \text{otherwise.}
    \end{cases}\]
\end{prop}

\begin{proof}
    The first case follows by definition of $e_{s(\alpha)}$. Suppose $\beta = \alpha\beta'$, $\beta' \not = \omega$. Then $r(\beta') \in s(\alpha)$ and 
\[ 
s_\alpha^*s_\beta = s_\alpha^*s_\alpha s_{\beta'} = e_{s(\alpha)}s_{\beta'} = s_{\beta'}.\] 

Likewise, suppose $\alpha = \beta \alpha'$, $\alpha' \not = \omega$. Then $r(\alpha') \in s(\beta)$ and $s_\alpha^*s_\beta = s_{\alpha'}^*s_{\beta}^*s_\beta = s_{\alpha'}^*e_{s(\beta)} = (e_{s(\beta)}s_{\alpha'})^* = s_{\alpha'}^*$.
    
Now suppose $\alpha$ is not a sublabel of $\beta$ and $\beta$ is not a sublabel of $\alpha$. Then for some labels $w,q,q' \in \mcL_S^*$ and $x,x' \in \mcL_S^1$ we have $\alpha = wxq, \beta = wx'q',$ and $x\not =x'$. Assume $w,q,q' \not = \omega$ (the other cases follow similarly). Then 
    \[s_\alpha^*s_\beta = s_q^*s_x^*s_w^*s_ws_{x'}s_{q'} \le s_q^*s_x^*s_{x'}s_{q'} \]
Now $s_x^*s_{x'} = 0$, by Proposition \ref{prop:threeproperties} (1), so $s_\alpha^*s_\beta = 0$.
\end{proof}

% I'm thinking we can just refer to the proposition any time we need to cite this fact
%\begin{cor}
%    Let $S$ be a combinatorial inverse semigroup with zero equipped with a coherent, traversing partial order, let $T_S$ be the associated labelled graph inverse semigroup, and let $\{s_\alpha:\alpha \in \mcL_S^{\geq 1}\}$ be as defined above. We have for $\alpha,\beta \in \mcL^{\geq 1}_S$, then $s^*_\alpha s_\beta \not = 0$ if and only if $\alpha$ is a sublabel of $\beta$ or vice versa. 
%\end{cor}

\begin{prop}
For $\alpha,\beta \in \mcL^{\geq 1}$, we have $\alpha = \beta$ if and only if $e_\alpha = e_\beta$.
\end{prop}

\begin{proof}
    If $\alpha = \beta$, then $e_\alpha = e_\beta$. Conversely, suppose that $e_\alpha = e_\beta$. Then $0 \not = e_\alpha e_\beta = s_\alpha s_\alpha^*s_\beta s_\beta^*$, so $s^*_\alpha s_\beta \not = 0$. This implies $\alpha$ is a sublabel of $\beta$ or vice versa. For contradiction, assume $\alpha$ is a proper sublabel of $\beta$. That is, $\beta = \alpha x_1\ldots x_n$ for letters $x_1,\ldots, x_n \in \mcL^1$. Since $e_\alpha = e_\beta$, $s_\alpha s_\alpha^* = s_\beta s_\beta^*$. Conjugating by $s_\alpha^*$, we have
\begin{align*}
s_\alpha^*s_\alpha  &= s_\alpha^*(s_\beta s_\beta^*)s_\alpha \\
					&= (s_\alpha^*s_\beta)(s_\alpha^*s_\beta)^* \\
					&= s_{x_1}\ldots s_{x_n} s_{x_n}^*\ldots s_{x_1}^* \\
					&\leq s_{x_1}s_{x_1}^*.
\end{align*}    

Let $x_1 = (u,g) \in \mcL^1$ and $s_\alpha^*s_\alpha = e_{s(\alpha)} = e_v$ for some $u,v \in S/\D$ and $g \in E(S)$. We know $g \ll e_u$ and $g \not = e_{[g]}$. Notice that
\[ 0 \not  = e_v = s_\alpha^*s_\alpha \le s_{x_1}s_{x_1}^* = g \ll e_u,\]
so $g = e_{[g]}$ by Proposition \ref{prop:threeproperties} (2), a contradiction. Thus $\alpha$ is not a proper sublabel of $\beta$. Similarly it can be shown that $\beta$ is not a proper sublabel of $\alpha$. Thus $\alpha = \beta$.
\end{proof}

%Once again we assume that $S$ is a combinatorial inverse semigroup with $0$, $\C$ is a coherent set for $S$, and $(\E, \mcL, \B)$ is the associated labelled graph space. Let $S_\C$ be the labelled graph inverse semigroup of $(\E, \mcL, \B)$.

We now define a map $\sigma : S_\C \to S$ by
\[\sigma(\alpha, B_v, \beta) = \begin{cases}
    s_\alpha e_v s_\beta^* & \text{if } \alpha,\beta \not = \omega\\
    s_\alpha e_v & \text{if }\alpha \not = \omega = \beta\\
    e_v s_\beta^* & \text{if } \beta \not = \omega = \alpha\\
    e_v & \text{if } \alpha,\beta = \omega
\end{cases}\]
and $\sigma (0) = 0$. We will show that $\sigma$ is an injective homomorphism.

\begin{prop}\label{prop:ranges}
    For $\alpha \in \mcL^{\geq 1}$, we have that $e_\alpha \not \in \C$ and $e_\alpha < e_v$ for some $v \in \E^0_S$.
\end{prop}

\begin{proof}
    Let $\alpha = x_1\ldots x_n$, and let $x_1 = (v,f)$ where $f \not = e_{[f]}$. Then
\begin{align*}
e_\alpha &= s_\alpha s_\alpha^* \\
		 &= s_{x_1}\ldots s_{x_n}s_{x_n}^*\ldots s_{x_1}^* \\
		 &\leq s_{x_1}s_{x_1}^* = f.
\end{align*}    
    
    If $e_\alpha \in \C$, then $f \in \C$ since $0 \neq e_\alpha \le f \ll e_v$, by Proposition \ref{prop:threeproperties} (2). By contradiction, so $e_\alpha \not\in \C$
\end{proof}

\begin{prop}\label{prop:thezerocase}
   For $t \in S_\C$, $\sigma(t) = 0$ if and only if $t = 0$.
\end{prop}

\begin{proof}
    If $t = 0$, then $\sigma(t) = 0$ by defintion. Conversely, suppose $t = (\alpha, B_v,\beta) \in T_\C$ with $B_v \neq \emptyset$. We also supppose $\alpha, \beta \neq \omega$, as the cases where one or both are empty labels are similar. There exists $u,w \in \E^0$ such that $s(\alpha)  = B_w$ and $s(\beta) = B_u$. Moreover $v \in B_w \cap B_u$. That is, $e_v \leq e_w$ and $e_v \leq e_u$. Thus
\begin{align*}
\sigma(t)\sigma(t)^* &= s_\alpha e_v s_\beta^* s_\beta e_v s_\alpha^* \\
					 &=	s_\alpha e_v e_{s(\beta)} e_v s_\alpha^* \\
					 &= s_\alpha e_v e_u s_\alpha^* \\
					 &= s_\alpha e_v s_\alpha^*
\end{align*}
But then $s_\alpha^*\sigma(t)\sigma(t)^*s_\alpha = s_\alpha^* s_\alpha e_v s_\alpha^* s_\alpha = e_w e_v e_w = e_v \neq 0$. It follows that $\sigma(t) \neq 0$.
\end{proof}

\begin{prop}
    The map $\sigma: S_\C \to S$ is injective.
\end{prop}

\begin{proof}

Since $S_\C$ is combinatorial, it suffices to prove that $\sigma$ is injective on idempotents. We may assume both idempotents are nonzero by Proposition \ref{prop:thezerocase}. Let $(\alpha, B_v, \alpha)$ and $(\gamma, B_u, \gamma)$ be nonzero idempotents in $S_\C$ such that $\sigma(\alpha, B_v, \alpha) = \sigma(\gamma, B_u, \gamma)$. There are four cases to consider. First suppose $\alpha = \gamma = \omega$. Then $e_v = \sigma((\omega, B_v, \omega)) = \sigma((\omega, B_u, \omega)) = e_u.$ So $u = v$ and $(\omega, B_v, \omega) = (\omega, B_u, \omega)$. 

Next suppose that $\alpha = \omega \neq \gamma$. Then
\begin{align*}
	e_u &= \sigma(\omega, B_u, \omega) \\
		&= \sigma(\gamma, B_v, \gamma) \\
		&= s_{\gamma} e_v s_{\gamma}^* \\ 
		&\leq s_{\gamma} s_{\gamma}^* = e_{\gamma}. 
\end{align*}
We know by Proposition \ref{prop:ranges} that $e_{\gamma} < e_w$ for some vertex $w$ and $e_{\gamma} \not\in \C$. However, by Definition \ref{def:coherent}(2), $e_u \leq e_{\gamma} < e_w$ implies that $e_{\gamma} \in \C$, a contradiction. The case where $\alpha \neq \omega = \gamma$ is similar. Finally, suppose that $\alpha, \gamma \neq \omega$. We have 
\[s_\alpha e_v s_\alpha^* = s_\gamma e_u s_\gamma^*.\] 
Since $B_v \subseteq s(\alpha)$ and $B_u  \subseteq s(\gamma)$, we know that $e_v \leq e_{s(\alpha)} = s_{\alpha}^* s_{\alpha}$ and $e_u \le e_{s(\gamma)} = s_{\gamma}^* s_{\gamma}$. Notice that $s_\alpha e_v s_\alpha^* s_\gamma e_u s_\gamma^* = s_\alpha e_v s_\alpha^* \neq 0$. Thus $s_\alpha^* s_\gamma \neq 0$. By Proposition \ref{prop:pathprods}, either $\alpha$ is a subword of $\gamma$ or $\gamma$ is a subword of $\alpha$. We want to show that $\alpha = \gamma$. Suppose for the sake of contradiction that $\alpha = \gamma \alpha'$ where $\alpha' \not = \omega$. Then we have that
\[
e_u = s_\gamma^* s_\gamma e_u s_\gamma^* s_\gamma = s_\gamma^* s_\alpha e_v s_\alpha^* s_\gamma = s_{\alpha'} e_v s_{\alpha'}^* \leq s_{\alpha'} s_{\alpha'}^*.
\]

As earlier in this proof, this leads to the contradiction that $s_{\alpha'} s_{\alpha'}^* \in \C$. We arrive at a similar contradiction in the case that $\gamma = \alpha \gamma'$ where $\gamma' \not = \omega$. Therefore $\alpha = \gamma$. Finally
\[ e_u = s_\gamma^* s_\gamma e_u s_\gamma^* s_\gamma = s_\gamma^* s_\alpha e_v s_\alpha^* s_\gamma = s_\alpha^* s_\alpha e_v s_\alpha^* s_\alpha = e_v.\]
Thus $u = v$ and we have $(\alpha, B_v, \alpha) = (\gamma, B_u, \gamma)$.
\end{proof}

\begin{prop}
    The map $\sigma: S_\C \to S$ is a homomorphism.
\end{prop}
\begin{proof}
    Let $t_1 = (\alpha, B_v, \beta), t_2= (\gamma, B_u, \delta) \in S_\C$. We will examine the cases where $\alpha,\beta,\gamma, \delta \not =\omega$. The other cases follow similarly. First suppose that $\beta = \gamma$. Then 
    \begin{align*}
        \sigma(t_1)\sigma (t_2)& =\sigma(\alpha, B_v, \beta)\sigma (\gamma, B_u, \delta)\\
        & = s_\alpha e_v s_\beta^* (s_\gamma) e_u s_\delta^*\\
        %& = s_\alpha (e_v s_\beta^* s_\beta) e_u s_\delta^*\\
        & = s_\alpha( e_ve_u) s_\delta^*\\
        & = s_\alpha e_{v \wedge u} s_\delta^*\\
        & = \sigma(\alpha, B_{u \wedge v}, \delta)\\
        & = \sigma(\alpha, B_v \cap B_u, \delta)\\
        & = \sigma((\alpha, B_v, \beta) (\gamma, B_u, \delta)) = \sigma(t_1t_2).
    \end{align*}
Next suppose that $\gamma = \beta \gamma', \gamma' \not = \omega,$ and $r(\gamma') \in B_v$. Let $s(\gamma) = B_w$ where $w \in S/\D$.
    \begin{align*}
        \sigma(t_1)\sigma (t_2)& =\sigma(\alpha, B_v, \beta)\sigma (\gamma, B_u, \delta)\\
        & = s_\alpha e_v s_\beta^* (s_\gamma) e_u s_\delta^*\\
        %& = s_\alpha e_v (s_\beta^* s_\beta) s_{\gamma'} e_u s_\delta^*\\
        & = s_\alpha (e_v s_{\gamma'}) e_u s_\delta^*\\
        & = s_\alpha (s_{\gamma'}) e_u s_\delta^*\\
        & = (s_\alpha s_{\gamma'}) e_{s(\gamma')} e_u s_\delta^*\\
        & = s_{\alpha\gamma'}(e_we_u) s_\delta^*\\
        %& = s_{\alpha\gamma'}e_{w \wedge u} s_\delta^*\\
        & = \sigma (\alpha\gamma', B_{w \wedge u}, \delta)\\
        & =  \sigma (\alpha\gamma', B_w\cap B_u, \delta)\\
        & =  \sigma (\alpha\gamma', s(\gamma')\cap B_u, \delta)\\
        & = \sigma((\alpha, B_v, \beta) (\gamma, B_u, \delta)) = \sigma(t_1t_2).
    \end{align*}
The next case, $\beta = \gamma\beta', \beta' \not = \omega,$ and $r(\beta') \in B_u$, is similar to the previous case. Finally, suppose that $\beta$ and $\gamma$ are not comparable. Then 
    \[\sigma(t_1)\sigma (t_2)=\sigma(\alpha, B_v, \beta)\sigma (\gamma, B_u, \delta)=s_\alpha e_v (s_\beta^* s_\gamma) e_u s_\delta^* = 0 \]
    and \[\sigma(t_1t_2) = \sigma((\alpha, B_v, \beta) (\gamma, B_u, \delta)) = \sigma(0) = 0.\]
\end{proof}

\subsection{$S$ is Morita equivalent to $S_\C$}
We continue with our standing assumptions that $S$ is a combinatorial inverse semigroup with $0$ having finite intervals and admitting a coherent set $\C$ of idempotent representatives of the $\D$-classes of $S$.

Our next goal is to show that $S$ is Morita equivalent to the labelled graph inverse semigroup $S_\C$. We accomplish this by showing that $S$ is an enlargement of $\sigma(S_\C)$. Recall that a coherent set $\C = \{ e_u : u \in S/\D \}$ on $S$ induces a partial order $\fl$ on $S/\D$ where $u \fl v$ if and only if $e_u \leq e_v$. Moreover $(S/\D, \fl)$ is a meet semilattice with $u \wedge v = [e_u e_v]$.

Throughout we let $\sigma : S_\C \to S$ be the embedding of the inverse semigroup of the labelled graph space $(\E, \mcL, \B)$ into $S$ defined earlier in this section. We will show that 
\[ \sigma(S_\C) = \C S \C.
\]
Let $f \in E(\C S \C)$. Since $f \D e_{[f]}$, there exists an $s_f \in S$ s.t. $s_f^*s_f = f$ and $s_fs_f^* = e_{[f]}$. We can also define an isomorphism $\tau_{f} : fSf \to e_{[f]}Se_{[f]}$ by $\tau_f(x) = s_fxs_f^*$.

%\begin{prop}
%    Suppose $s \le t$ with $e \le s^*s$, then $ses^* = tet^*$.
%\end{prop}
%
%\begin{proof}
%    \[t(e)t^* = (ts^*s)e(s^*st^*) = ses^*\]
%\end{proof}

\begin{lem}\label{lem:sametau}
    Let $g \le f \le e$ for $g,f,e \in E(\C S \C)$. If $\tau_e(f) = e_{[f]}$, then $\tau_e(g) = \tau_f(g)$. 
\end{lem}

\begin{proof}
    Let $t = s_e s_f^* s_f = s_e f$. Then $t^*t = fs_e^*s_ef = fef = f = s_f^*s_f$ and $t t^* = s_effs_e^* = s_efs_e^* = \tau_e(f) = e_{[f]} = s_fs_f^*$. Since $S$ is combinatorial, $s_f = t = s_e s_f^* s_f$. Thus $s_f \leq s_e$. It follows that 
	\[
		\tau_e(g) = s_e g s_e^* = s_e (s_f^* s_f) g (s_f^* s_f) s_e^* = s_f g s_f^* = \tau_f(g).
	\]

\end{proof}

\begin{prop}\label{prop:chain} For every idempotent $e \in E(\C S \C)$, there exists $u \in S/\D$ and $f_1,g_1,\ldots, f_n,g_n \in E(\C S \C)$ such that 
\[
e = f_1\le g_1 \ll f_2 \le g_2 \ll \cdots \ll f_n\le g_n = e_u, 
\]
where $\tau_{g_i}(f_i) = e_{[f_i]}$ for $1 \le i \le n$ and $\tau_{f_{i+1}}(g_i) \not = e_{[g_i]}$ for $1 \le i \le n-1$.
\end{prop}

\begin{proof}
    The proposition holds trivially in the case that $e = 0$, since $e = 0 \leq 0 = e_{[0]}$ and $\tau_0(0) = 0$. Let $0 \neq e \in E(\C S \C)$. Then $e = e_{u_1}se_{u_2}$ where $u_1,u_2 \in S/\D$ and $s \in S$. Notice $0 \neq e \leq e_u$ where $u = u_1 \wedge u_2$. Since $S$ admits finite intervals, there exists idempotents $h_i$ for $1 \leq i \leq m$ such that 
    \[ 0 \neq e = h_1 \ll h_2 \ll \dots \ll h_m = e_u. \]
We proceed by induction on $m$. For $m=1$ we have $e = e_u$. In this case, let $f_1 = g_1 = e = e_u$ and notice that $\tau_{g_1}(f_1) = \tau_{e}(e) = e = e_u = e_{[f_1]}$. Next suppose that claim holds for any chain of idempotents of length $k$ and that $m = k+1$. By the inductive hypothesis applied to the chain $h_2 \ll \dots \ll h_m = e_u$, there exists $f_i, g_i$ for $2 \leq i \leq n$ such that
\[
	h_2 = f_2 \leq g_2 \ll f_3 \leq g_3 \ll \dots \ll f_n \leq g_n = e_u,
\]
where $\tau_{g_i}(f_i) = e_{[f_i]}$ for $2 \le i \le n$ and $\tau_{f_{i+1}}(g_i) \not = e_{[g_i]}$ for $2 \le i \le n-1$. We consider two cases. First suppose that $\tau_{h_2}(e) \neq e_{[e]}$. In this case we let $f_1 = g_1 = e$. Then $e = f_1\le g_1 \ll f_2 \le g_2 \ll \cdots \ll f_n\le g_n = e_u$. Notice $\tau_{g_1}(f_1) = e_{[e]} = e_{[f_1]}$, $\tau_{f_2}(g_1) = \tau_{h_2}(e) \neq e_{[e]} = e_{[g_1]}$, and the remaining conditions follow from the chain obtained by the inductive hypothesis.

Next, suppose that $\tau_{h_2}(e) = e_{[e]}$. In this case, define $f'_1 = e$, $g'_1 = g_2$, $f'_i = f_{i+1}$ and $g'_i = g_{i+1}$ for $2 \leq i \leq n-1$. We have
\[
	e = f'_1 \leq g'_1 \ll f'_2 \leq g'_2 \ll \dots \ll f'_{n-1} \leq g'_{n-1} = e_u.
\]
Notice that $e = f'_1 \leq h_2 \leq g_2 = g'_1$. Since $\tau_{g_2}(h_2) = \tau_{g_2}(f_2) = e_{[f_2]}$, we have $\tau_{g'_{1}}(f'_1) = \tau_{g_{2}}(e) = \tau_{h_2}(e) = e_{[e]}$ by Lemma \ref{lem:sametau}. Also, $\tau_{f'_{2}}(g'_1) = \tau_{f_{3}}(g_2) \neq e_{[g_2]} = e_{[g'_1]}$. The remaining conditions hold for the chain obtained in the inductive hypothesis. By induction, this completes the proof.
\end{proof}

\begin{prop}\label{prop:words} For all $e \in E(\C S \C) - \C$, there exists $\alpha \in \mcL^{\geq 1}$, such that $e = s_\alpha e_{[e]} s_\alpha^*$ and $e_{[e]} \leq s_{\alpha}^* s_{\alpha}$.
\end{prop}

\begin{proof}
    Let $e \in E(\C S \C) - \C$, thus $e \not =0 $ and $e_{[e]} \not = 0$. By Proposition \ref{prop:chain}, we know there exists $u \in S/D$, $f_1,g_1,\ldots, f_n,g_n \in \C S \C$ such that 
    \[e = f_1 \le g_1 \ll f_2 \le g_2 \ll \cdots \ll f_n \le g_n = e_u,\]
    where $\tau_{g_i}(f_i) = e_{[f_i]}$ for $i = 1,\ldots, n$ and $\tau_{f_{i+1}}(g_i) \not = e_{[g_i]}$ for $i = 1,\ldots, n-1$. We claim $n \geq 2$. If not, then $e = f_1 \le g_1 = e_u$. But $\tau_{g_1}(x) = \tau_{e_u}(x) = e_uxe_u$, and thus $e_{[e]} = \tau_{e_u}(e) = e_u e e_u = e_u e$. This implies that $0 \not =e_{[e]} \le e \le e_u$ and hence $e \in \C$, a contradiction.
  
So $n \geq 2$. Notice that since $g_i \ll f_{i+1}$, we have that 
    \[
    \tau_{f_{i+1}}(g_i) \ll \tau_{f_{i+1}}(f_{i+1}) = e_{[f_{i+1}]}
    \] 
for $1 \leq i \leq n-1$. Since $\tau_{f_{i+1}}(g_i) \not = e_{[g_i]}$, $x_i = ([f_{i+1}], \tau_{f_{i+1}}(g_i), [g_i]) \in \E^1$. Since $f_{i+1} \le g_{i+1},$ it follows that 
   \[ e_{[f_{i+1}]} = \tau_{g_{i+1}}(f_{i+1}) \le \tau_{g_{i+1}}(g_{i+1}) = e_{[g_{i+1}]}.
   \] 
Thus $[f_{i+1}] \le [g_{i+1}] \in B_{[g_{i+1}]}$. It follows that $x_{n-1}x_{n-2}\ldots x_1$ is a path of length $n-1$. Let $a_i = ([f_{i+1}], \tau_{f_{i+1}}(g_i))$ for $1 \leq i \leq n-1$. Then $\alpha = a_{n-1}a_{n-2}\ldots a_1 \in \mcL^{\geq 1}$.
    
We will now show by induction on $n$ that $e = s_\alpha e_{[e]}s_\alpha^*$. First suppose that $n=2$. Since $g_2 = e_u$ we have that $f_2 = \tau_{g_2}(f_2) = e_{[f_2]}$. Then $\alpha = ([f_2], \tau_{f_2}(g_1)) = ([f_2], g_1)$. Notice that $s_\alpha s_\alpha^* = g_1$ and $s_\alpha^*s_\alpha = e_{[g_1]}$, it follows that $s_{g_1} = s_\alpha^*$ and 
    \[e = f_1 = \tau^{-1}_{g_1}(e_{[f_1]}) = s_\alpha e_{[f_1]}s_\alpha^* = s_\alpha e_{[e]}s_\alpha^*.\]
    
Now suppose that $n \geq 3$ and write $\alpha = \alpha' a_1$. We have $f_2 = s_{\alpha'} e_{[f_2]} s_{\alpha'}^*$ by induction. We claim that $g_1 = s_\alpha s_\alpha^*$. First, let $t = e_{[f_2]}s_{\alpha'}^*$. We will show $t = s_{f_2}$. Indeed, $tt^* = e_{[f_2]}s_{\alpha'}^*s_{\alpha'} e_{[f_2]} = e_{[f_2]}$ and 
\[
t^*t = s_{\alpha'} e_{[f_2]} e_{[f_2]} s_{\alpha'}^* = s_{\alpha'} e_{[f_2]} s_{\alpha'}^* = f_2. 
\]
So $s_{f_2} = e_{[f_2]}s_{\alpha'}^*$. Thus
    \begin{align*}
        \tau_{f_2}(s_\alpha s_\alpha^*) & = s_{f_2} (s_\alpha s_\alpha^*) s_{f_2}^*\\
        & = e_{[f_2]}s_{\alpha'}^* (s_\alpha) (s_\alpha^*) s_{\alpha'} e_{[f_2]}\\
        & = e_{[f_2]}s_{\alpha'}^* s_{\alpha'}s_{a_1} s_{a_1}^* s_{\alpha'}^* s_{\alpha'} e_{[f_2]}\\
        & = e_{[f_2]}s_{a_1} s_{a_1}^* e_{[f_2]}\\
        & = e_{[f_2]}\tau_{f_2}(g_1) e_{[f_2]}\\
        & = \tau_{f_2}(g_1).
    \end{align*}
    So $g_1 = s_\alpha s_\alpha^*$ as desired. It also follows that $s_{g_1} = s_\alpha^*$. Therefore $e = f_1 = \tau_{g_1}^{-1}(e_{[f_1]}) = s_\alpha e_{[f_1]}s_\alpha^* = s_\alpha e_{[e]}s_\alpha^*$. We have shown that, $e = s_\alpha e_{[e]}s_\alpha^*$ for all $e \in E(\C S \C) - \C$. Finally, note that $e_{[e]} = e_{[f_1]} = \tau_{g_1}(f_1) \leq e_{[g_1]} = s_{\alpha}^* s_{\alpha}$.
\end{proof}

%\begin{prop} Suppose that $S$ is a combinatorial inverse semigroup with $0$ admitting a coherent set $\C$. Then $\sigma(S_\C) \subseteq \C S \C$.
%\end{prop}
%
%\begin{proof}
%    Let $(\alpha, B_v, \beta) \in S_\C$. Suppose $\alpha,\beta \not = \omega$. The other cases have similar proofs. Let $t = \sigma(\alpha,B_v,\beta) = s_\alpha e_v e_\beta^*$. Then 
%    \[t^*t = s_\beta e_v s_\alpha^* s_\alpha e_v e_\beta^* \le s_\beta s_\beta^* = e_\beta \le e_{r(\beta)}\]
%    Similarly, $tt^* \le e_{r(\alpha)}$. Thus $\sigma(\alpha,B_v,\beta) = t = e_{r(\alpha)} t e_{r(\beta)} \in \C S \C$.
%\end{proof}

\begin{thm}
    For the embedding $\sigma : S_\C \to S$, we have $\sigma(S_\C) = \C S \C$.
\end{thm}

\begin{proof}
Let $(\alpha, B_v, \beta) \in S_\C$. Suppose $\alpha,\beta \not = \omega$. The other cases have similar proofs. Let $t = \sigma(\alpha,B_v,\beta) = s_\alpha e_v s_{\beta}^*$. Then 
    \[t^*t \leq s_\beta s_\beta^* = e_\beta \leq e_{r(\beta)}.\]
Similarly, $tt^* \leq e_{r(\alpha)}$. Thus $\sigma(\alpha,B_v,\beta) = t = e_{r(\alpha)} t e_{r(\beta)} \in \C S \C$. So $\sigma(S_\C) \subseteq \C S \C$.

Next, let $t \in \C S \C$. If $t = 0$, then $t = \sigma(0)$. Otherwise, there are four cases to consider. First suppose $t^*t \in \C$ and $tt^* \in \C$. Since $\C$ contains exactly one representative from each $\mathcal{D}$-class, $t^*t = t t^* = e_{[t]}$. Thus $t = e_{[t]} \in \C S \C$. 

Suppose that $t^*t \in \C$ and $tt^* \not \in \C$. Then $t^* t = e_{[t^* t]}$ and by Proposition \ref{prop:words}, $t t^* = s_{\alpha} e_{[tt^*]} s_{\alpha}^*$ for some $\alpha \in \mcL^{\geq 1}$ and $e_{[t t^*]} \leq s_{\alpha}^* s_{\alpha}$. Notice that $(s_{\alpha} e_{[tt^*]})(s_{\alpha} e_{[tt^*]})^* = s_{\alpha} e_{[tt^*]} s_{\alpha}^* = tt^*$ and 
\[
(s_{\alpha} e_{[tt^*]})^* (s_{\alpha} e_{[tt^*]})= e_{[tt^*]} s_{\alpha}^* s_{\alpha} = e_{[tt^*]} = e_{[t^*t]} = t^*t.
\]
Thus $t = s_{\alpha} e_{[tt^*]} = \sigma( (\alpha, B_{[tt^*]}, \omega))$, as desired. The proof is very similar in the case that $t^*t \not \in \C$ and $tt^* \in \C$.

Finally, suppose that $t^*t \not \in \C$ and $tt^* \not \in \C$. Then there exists $\alpha, \beta \in \mcL^{\geq 1}$ such that $tt^* = s_\alpha e_{[tt^*]} s_\alpha^*$ and $t^* t = s_\beta e_{[t^*t]} s_\beta^*$ with $e_{[t^*t]} \leq s_{\beta}^* s_{\beta}$ and $e_{[tt^*]} \leq s_{\alpha}^* s_{\alpha}$. Also notice $e_{[t^*t]} = e_{[tt^*]} = e_{[t]}$. We then have 
\[
    \sigma(\alpha, B_{[t]}, \beta) \sigma(\alpha, B_{[t]}, \beta)^* = s_\alpha e_{[t]} s_\beta^* s_\beta e_{[t]}s_\alpha^* = s_\alpha e_{[t]} s_\alpha^*= tt^*
\]
and similarly $\sigma(\alpha, B_{[t]}, \beta)^* \sigma(\alpha, B_{[t]}, \beta) = t^*t$. Therefore $t = \sigma(\alpha, B_{[t]}, \beta)$.
\end{proof}

Using Proposition \ref{prop:enlargement}, we now have the following result.

\begin{thm}\label{thm:labelledgraph} Suppose that $S$ is a combinatorial inverse semigroup with $0$ having finite intervals and admitting a coherent set $\C$ of idempotent representatives of the $\D$-classes of $S$. Then $S$ is Morita equivalent to the labelled graph inverse semigroup $S_\C$.
\end{thm}

\section{Inverse Hulls, Cores, and Labelled Graphs}

We start this section with a brief review of the inverse semigroup associated to a Markov shift by Starling \cite{StarlingShifts}. We will rely on a number of structural properties proved in \cite{MarkovShift}. Those properties are summarized below. We define a partial order on the $\D$-classes of such semigroups using the idea of a core set of idempotents. Given a finite alphabet $A$ and a matrix $T = \{T_{a,b}\}_{a,b \in A}$ with $T_{a,b} \in \{0,1\}$ for each $a,b \in A$, such that each row of $T$ has a nonzero entry, we refer to $T$ as a \emph{Markov transition matrix}. Associated to $T$ is a semigroup $L_T \cup \{0\}$, where $L_T$ is the set of allowable words defined by $T$ and multiplication is given by 
\[
    x*y = \begin{cases}
            xy & \text{if $xy \in L_T$} \\
            0 & \text{otherwise}
    \end{cases}
\]
For each $a \in A$ there is a partial bijection
\[
    \theta_a : \{w \in L_T : aw \in L_T\} \to \{v \in L_T : v = aw \text{ for $w \in L_T$} \} 
\]
defined by $\theta_a(w) = aw$. The \emph{inverse hull of $T$} is the inverse semigroup $H(T)$ generated by the partial bijections $\{ \theta_a : a \in A\} \cup \{0\}$. The following is proved in \cite{MarkovShift}.

\begin{prop}
    Let $T$ be a Markov transition matrix over a finite alphabet $A$. Then $H(T)$ is a combinatorial inverse semigroup with $0$. Moreover, the set of idempotents $\mcO = \{\theta_a\theta_a^{-1}: a \in A \}$ satisfies the following properties:
    \begin{itemize}
        \item [(O1)] the elements of $\mcO$ are mutually orthogonal,
        \item [(O2)] every idempotent in $H(T)$ is comparable to some element of $\mcO$,
        \item [(O3)] both $\mcO^\uparrow \cup \{0\}$ and $(\mcO^\uparrow - \mcO) \cup \{0\}$ are closed under multiplication,
        \item [(O4)] elements of $\mcO^\uparrow - \mcO$ are uniquely determined by the set of idempotents in $\mcO$ that they lie above in the natural partial order, and
        \item [(O5)]  for each $e \in \mcO$, the $\D$-class of $e$ contains at most one element of $\mcO^\uparrow - \mcO$.
    \end{itemize}
\end{prop}

We also have that 
\[
(\mcO^\uparrow - \mcO) = \{\theta_{a_1}^{-1}\theta_{a_1}\cdots\theta_{a_n}^{-1}\theta_{a_n}: a_i \in A \text{ for $1 \leq i \leq n$}\}.
\]
Every element in $H(X)$ can be written in the form 
\[
\theta_s\theta_{x_1}^{-1}\theta_{x_1}\cdots\theta_{x_n}^{-1}\theta_{x_n}\theta_w^{-1},
\]
for some $n\geq 1$, $x_i \in A$, and $s,w \in L_T^1$. Such an element is idempotent if $s = w$. 
%where if $s \in L_X$, then $s = s'x_1$ where $s' \in L_X^1$ and if $w \in L_X$, then $w = w'x_n$ where $w' \in L_X^1$. 
\begin{prop}
    For every nonzero $\D$-class of $H(T)$, $\mcO^\uparrow - \mcO$ contains exactly one element in that class.
\end{prop}
\begin{proof}
    What need to show that if
    \[
    e = \theta_{x_1}^{-1}\theta_{x_1}\cdots\theta_{x_n}^{-1}\theta_{x_n} \; \D \; \theta_{y_1}^{-1}\theta_{y_1}\cdots\theta_{y_m}^{-1}\theta_{y_m} = f
    \]
then $e = f$. Choose $s$ such that $s^*s = e$ and $ss^* = f$ and write
\[
    s = \theta_v\theta_{z_1}^{-1}\theta_{z_1}\cdots\theta_{z_k}^{-1}\theta_{z_k}\theta_w^{-1}.
\]
Then
    \[\theta_{x_1}^{-1}\theta_{x_1}\cdots\theta_{x_n}^{-1}\theta_{x_n} = \theta_w\theta_{z_1}^{-1}\theta_{z_1}\cdots\theta_{z_k}^{-1}\theta_{z_k}\theta_w^{-1},\]
    \[\theta_{y_1}^{-1}\theta_{y_1}\cdots\theta_{y_m}^{-1}\theta_{y_m} = \theta_v\theta_{z_1}^{-1}\theta_{z_1}\cdots\theta_{z_k}^{-1}\theta_{z_k}\theta_v^{-1}.\]
    So we see that $v = w = 1$. And thus
    \[\theta_{x_1}^{-1}\theta_{x_1}\cdots\theta_{x_n}^{-1}\theta_{x_n} = \theta_{z_1}^{-1}\theta_{z_1}\cdots\theta_{z_k}^{-1}\theta_{z_k} = \theta_{y_1}^{-1}\theta_{y_1}\cdots\theta_{y_m}^{-1}\theta_{y_m},\]
    as desired.
\end{proof}

We want to define a partial order on the $\D$-classes of the inverse hull of a Markov shift that will later be used to define a labelled graph. We need only look at the set of idempotents in $(\mcO^\uparrow - \mcO) \cup \{0\}$ to determine this relation as this set already contains a representative of each nonzero $\D$-class. The partial order we define will be induced by some of the relations in the natural partial order on $\mcO^\uparrow - \mcO$. The essential idea is to include those relations that appear between idempotents in certain diamonds, while excluding the others. We formalize this idea using the notion of a core set of idempotents.

\begin{defn}\label{def:core} Let $S$ be an inverse semigroup with $0$ that has finite intervals. For $e \in E(S)$ we define $\C_e$ to be the collection of subsets $C\subseteq e^\downarrow$ satisfying:
\begin{itemize}
    \item [(1)] $e \in C$.
    \item [(2)] If $f,g \in C$ with $fg \not = 0$, then $fg \in C$.
    \item [(3)] For each $h_1,h_2 \in C$ with $f \ll h_1$ and $g \ll h_2$, if $f$ and $g$ are incomparable and $fg \not = 0$, then $f,g \in C$.
    \item [(4)] For $e_1, e_2$ in $C$, if $e_1 < g< e_2$, then $g \in C$.
\end{itemize}
Then we call $C_e = \bigcap_{C \in \C_e}C$, the \textit{core} of $e$.
\end{defn}

One can quickly verify that $C_e \in \C_e$. We include here a useful fact for later that can be proved by a quick induction argument using property (3).

\begin{lem}\label{lem:core} Let $S$ be an inverse semigroup with $0$ having finite intervals. Suppose that $e,f,g$ are nonzero idempotents with $f \ll e$ and $g < e$ such that $f$ and $g$ are incomparable and $fg \neq 0$. Then $f,g \in C_e$.
\end{lem}

\begin{defn} For $a,b \in S /\D$, we write $a \fl' b$ if there exists $e \in E(S)$ and $f,g \in C_e$ where $[f] = a$ and $[g] = b$ and $f \leq g$. It is reflexive since if $a = [e]$, then $e \in C_e$ and hence $a \fl' a$. Let $\fl$ be the transitive closure of $\fl'$ together with the additional relations that $0 \fl a$ for every $a \in S/\D$. 
\end{defn}

We will prove that $\fl$ is antisymmetric when $S$ is the inverse hull of a Markov shift. First we show that this relation is preserved under Morita equivalence. Recall that if $S$ and $T$ are Morita equivalent then there is an equivalence functor $\F : C(S) \to C(T)$. Moreover, the map $\sigma([e]) = [\F(e)]$ is a well-defined bijection from the $\D$-classes of $S$ to the $\D$-classes of $T$ (see Lemma \ref{lem:equivalence}).

\begin{prop}
    Let $S$ and $T$ be inverse semigroups with $0$. If $S$ is Morita equivalent to $T$ via an equivalence functor $\F:C(S) \to C(T)$, then $a \fl_S b$ if and only if $\sigma(a) \fl_T \sigma(b)$ where $\sigma([e]) = [\F(e)]$.
\end{prop}

\begin{proof}
    It suffices to show that $a \fl'_S b$ if and only if $\sigma(a) \fl'_T \sigma(b)$. Suppose $a,b \in S/\D$ with $a \fl_S'b$. Let $e,f,g \in E(S)$ where $f,g \in C_e, [f] = a$, and $[g] = b$ and $f \leq g$. Choose any $e' \in E(T)$ such that $\sigma([e]) = [e']$. Let $\tau : eSe \to e'Te'$ be the isomorphism defined in Lemma \ref{lem:equivalence}. Then $\tau(f), \tau(g) \in C_{e'}$ with $\tau(f) \leq \tau(g)$. By part (2) of the lemma we have 
    \[
        \sigma(a) = [ \F(f) ] = [ \tau(f) ] \fl'_{T} [ \tau(g) ] = [ \F(g) ] = \sigma(b).
    \]
    The converse is similar.
\end{proof}

\begin{lem}\label{lem:ineq} Let $S$ be the inverse hull of a Markov shift. For $a, b \in S / \D$, if $0 \neq a \fl b$ then there exists $f, g \in \mcO^\uparrow - \mcO$ with $[f] = a, [g] = b$ and $f \leq g$.
\end{lem}
\begin{proof}
For any $e \in \mcO^\uparrow - \mcO$, one can show that $(\mcO^\uparrow - \mcO )\cap e^\downarrow$ satisfies the four properties in Definition \ref{def:core}. Thus $C_e \subseteq \mcO^\uparrow - \mcO$. Suppose $0 \neq a \fl' b$. Since every $\D$-class of $S$ has a unique representative in $\mcO^\uparrow - \mcO$, we may choose $e,f,g \in \mcO^\uparrow - \mcO$ such that $f,g \in C_e$, with $[f] = a, [g] = b,$ and $f \leq g$. Now suppose $0 \neq a \fl b$. Then $a = a_0 \fl' a_1 \fl' a_2 \fl' \dots \fl' a_n = b$ for some $\D$-classes $a_0, a_1, \dots a_{n}$. For each $1 \leq k \leq n$ we may choose idempotents $f_k, g_k$ in $\mcO^\uparrow - \mcO$ such that $[f_k] = a_{k-1}, [g_k] = a_{k}$ and $f_k \leq g_k$. Since each nonzero $\D$-class has a unique representative in $\mcO^\uparrow - \mcO$, we conclude that $g_k = f_{k+1}$ for $1 \leq k \leq n-1$. Thus $[f_1] = a, [g_n] = b$ and $f_1 \leq g_n$. 

\end{proof}

\begin{prop} Let $S$ be the inverse hull of a Markov shift. Then $\fl$ is a partial order making $(S/\D, \fl)$ into a meet semilattice.
\end{prop}
\begin{proof}
We know that $\fl$ is reflexive and transitive. We first show that $\fl$ is antisymmetric. Suppose $a \fl b$ and $b \fl a$. If either $a$ or $b$ is $0$ then $a = b = 0$ since the only relation of the form $c \fl 0$ is when $c = 0$. We may then assume $a,b \neq 0$. By Lemma \ref{lem:ineq} there exists $f_1, f_2, g_1, g_2$ in $\mcO^\uparrow - \mcO$ such that $[f_1] = a, [g_1] = b, [f_2] = b, [g_2] = a$ where $f_1 \leq g_1$ and $f_2 \leq g_2$. Since $\D$-classes in $\mcO^\uparrow - \mcO$ have unique representatives, we have $f_1 \leq g_1 = f_2 \leq g_2 = f_1$. Thus $f_1 = f_2$ and $a = b$. 

Next we show that $(S/\D, \fl)$ has meets. Suppose that $a, b \in S / \D$. If there is no nonzero $\D$-class $c$ such that $c \fl a,b$, then $a \wedge b = 0$. Suppose there exists a nonzero $\D$-class $c$ such that $c \fl a,b$. Then there are idempotents $e_a, e_b,$ and $e_c$ in $\mcO^\uparrow - \mcO$ such that $e_c \leq e_a, e_b$. We wish to show that $a \wedge b = [e_a e_b]$. We first prove that $[e_a e_b] \fl a$. If $e_c = e_a e_b$ we are done, so suppose that $e_c < e_a e_b$. Find $g_k$ in $\mcO^\uparrow - \mcO$ such that $e_c = g_n \ll \cdots \ll g_1 \ll g_0 = e_a$ where $[g_{k+1}] \fl' [g_{k}]$ for $1 \leq k \leq n-1$. Choose the largest $k$ so that $e_a e_b \leq g_k$. Notice $k < n$ and that $e_a e_b$ and $g_{k+1}$ are incomparable. It follows from Lemma \ref{lem:core} that $e_a e_b \in C_{g_k}$. So 
\[
    [e_a e_b] \fl' [g_k] \fl [e_a] = a.
\]
A similar proof shows that  $[e_a e_b] \fl b$. We now show that $c \fl [e_a e_b]$. Let $h_k = g_k e_b$ for $0 \leq k \leq n$. Suppose for the sake of contradiction that $c \not\fl [h_k]$ for some $k$ and choose the largest such $k$. Note $k < n$. Notice that $h_k g_{k+1} = h_{k+1} \neq 0$. Also, since $c \fl [g_k]$, we must have $h_k < g_k$. Now $g_{k+1}$ and $h_k$ are incomparable. Indeed, $g_{k+1} \leq h_k$ implies that $g_{k+1} \leq h_k \leq g_k$. But then $h_k = g_k$ or $h_k = g_{k+1}$ as $g_{k+1} \ll g_k$. This is impossible since $c \fl [g_j]$ for each $j$. Alternatively, $h_k \leq g_{k+1}$ implies that $h_{k} = h_{k} g_{k+1} = h_{k+1}$ contradicting the maximality of $k$. Since $g_{k+1}$ and $h_k$ are incomparable it follows that $g_{k+1}, h_k \in C_{g_k}$. But then $h_{k+1} = h_k g_{k+1} \in C_{g_k}$. Thus $c \fl [h_{k+1}] \fl' [h_k]$, a contradiction. Thus $c \fl [h_0] = [e_a e_b]$.

Thus we have shown that $(S/\D,\fl)$ is a meet semilattice where $a \wedge b = 0$ if there is no nonzero $c$ such that $c \fl a,b$ and otherwise $a \wedge b = [e_a e_b]$ where $e_a, e_b$ are $\D$-class representatives of $a$ and $b$ respectively chosen from $\mcO^\uparrow - \mcO$.
\end{proof}

\section{Morita Equivalence of Markov Shifts}

By the last two propositions, we know that if $S$ is the inverse hull of a Markov shift and we choose the set of $\D$-class representatives of $S$ to be $\{e_a : a \in S/\D\} = (\mcO^\uparrow - \mcO) \cup \{0\}$, then the induced partial order $\fl$ satisfies the conditions given in section 2. Thus we have an inverse semigroup $S^{\fl}$ that is Morita equivalent to $S$. We will now show that $S^{\fl}$ is Morita equivalent to a labelled graph inverse semigroup. 

\begin{prop}\label{prop:flmarkov} Let $S$ be the inverse hull of a Markov shift and
\[\{e_a : a \in S/\D \} =  (\mcO^\uparrow - \mcO) \cup \{0\}\] 
be a set of idempotent $\D$-class representatives for $S$. We have the following
\begin{enumerate}
\item if $a,b \in S / \D$ with $a \wedge b \neq 0$, then $e_a e_b = e_{a \wedge b}$, and
\item if $a,b,c \in S / \D$ with $a,b \fl c$, then $e_a e_b = e_{a \wedge b}$.
\end{enumerate}
\end{prop}
\begin{proof}
Let $a,b \in S / \D$ with $a \wedge b \neq 0$. Then $a \wedge b = [e_a e_b]$. Therefore we have that $e_{a \wedge b} = e_{[e_a e_b]} = e_a e_b$.

Next suppose $a,b,c \in S / \D$ with $a,b \fl c$. If $e_a e_b = 0$ then $a \wedge b = 0$ and we are done. Otherwise we can find $g_i$ such that $e_a = g_n \ll \cdots \ll g_1 \ll g_0 = e_c$ and $[g_{i+1}] \fl' [g_{i}]$ for $0 \leq i \leq n-1$. Also there exists $f_j$ such that $e_b = f_m \ll \cdots \ll f_1 \ll f_0 = e_c$ and $[f_{j+1}] \fl' [f_{j}]$ for $0 \leq i \leq m-1$. Without loss of generality, we can assume that $f_j$ and $g_i$ are incomparable for $i,j >0$ (we could replace $c$ with $c' = [f_i g_j] \fl c$ where $i$ and $j$ are appropriately chosen if necessary). Notice then that for all $i,j$ nonzero we have $g_i$ and $f_j$ incomparable with nonzero product. Using Lemma \ref{lem:core} we have that $e_a, e_b \in C_{e_c}$. Thus $e_a e_b \in C_{e_c}$ and $0 \neq [e_a e_b] \fl a, b$. Thus
\[
    e_{a \wedge b} = e_{[e_a e_b]} = e_a e_b.
\]
\end{proof}

\begin{prop}\label{prop:CDineq} Let $S$ be the inverse hull of a Markov shift and let $S^{\fl}$ be as defined above. Then 
$\C = \{ [ a, e_a, a] : e_a \in \mcO^\uparrow - \mcO\} \cup \{0\}$ is a coherent set for $S^{\fl}$. Moreover, $(\C, \leq)$ is order isomorphic to $(S/\D, \fl)$.
\end{prop}
\begin{proof}
We first show that the map $\sigma: S/\D \to \C$ where $\sigma(a) = [a, e_a, a]$ is an order isomorphism. Since $[a,e_a,a] = [b, e_b, b]$ if and only if $e_a = e_b$ it follows that $\sigma$ is a bijection. Next, suppose that $a \fl b$ with $a \neq 0$. By Lemma \ref{lem:ineq}, we have $e_a \leq e_b$. Since $e_a \leq e_b$ and $a \wedge b = a \neq 0$, it follows by Proposition \ref{prop:flineq} that $[a, e_a, a] \leq [b, e_b, b]$. Conversely, suppose that $0 \neq [a, e_a, a] \leq [b, e_b, b]$. Then $e_a \leq e_b$ and $a \wedge b \neq 0$. By Proposition \ref{prop:flmarkov} (1),  $e_{a \wedge b} = e_a e_b = e_a$. Thus $a \fl b$.

Next we prove that $\C$ satisfies the three properties given in Definition \ref{def:coherent}. Let $[a, e_a, a], [b, e_b, b]$. Since $0 \in \C$, we may assume that
\[ [a,e_a,a][b,e_b,b] = [a, e_a e_{a \wedge b} e_b, b] \neq 0.
\]
Thus $a \wedge b \neq 0$ and
\[
[a, e_a e_{a \wedge b} e_b, b] = [a, e_{a \wedge b}, b] = [a \wedge b, e_{a \wedge b}, a \wedge b]
\]
where the last equality follows from Proposition \ref{prop:flineq}. Thus $\C$ is closed under multiplication.

Next suppose that $0 \neq [a, e_a, a] \leq f \leq [b, e_b, b]$, for some idempotent $f \in S^{\fl}$. Then $f = [c, e, c]$ for some idempotent $e \leq e_c$. Moreover, $0 \neq e_a \leq e \leq e_b$. It follows that $e \in \mcO^\uparrow - \mcO$. Write $e = e_d$. We have $0 \neq e_a \leq e_d \leq e_b$ with $a \wedge b \neq 0$. Then $e_{a \wedge b} = e_a e_b = e_a$, so $a \fl b$. It follows by an argument similar to the proof of Proposition \ref{prop:flmarkov} (2), that $d \fl b$. Moreover, we have that $[c,e,c] = [b \wedge c, e, b \wedge c]$ since $b \wedge c \neq 0$. Also $a \wedge b \wedge c = a \wedge c \neq 0$. So, replacing $c$ with $b \wedge c$ if necessary, we may assume that $c \fl b$.

Since $c,d \fl b$, we have $e_{c \wedge d} = e_c e_d = e_d$. So $d \fl c$. It follows that $f = [c,e_d,c] = [d,e_d,d]$. Thus $f \in \C$.

Finally, suppose that $[a,f,a], [b,g,b] \ll [c, e_c, c]$ with $[a,f,a]$ and $[b,g,b]$ incomparable. We show that 
\[[a, f e_{a \wedge b} g, b] \in \C.
\]
If $[a, f e_{a \wedge b} g, b] = 0$ then we are done, so we may assume that $fg \neq 0$ and $a \wedge b \neq 0$. Then $[a,f,a] \neq 0$ and hence $a \wedge c \neq 0$. We have $[a, f e_{a \wedge b} g, b] = [a, fg, b]$. By replacing $a$ with $a \wedge c$ if necessary, we may assume that $a \fl c$. Similarly, we may assume that $b \fl c$.

We have $fg \neq 0$ with $f, g \ll e_c$ and $a \wedge b \neq 0$. If $f = g$ then $[a,f,a] = [b,g,b]$, which contradicts the assumption that the two elements are incomparable. Thus $f,g$ are incomparable (since each is immediately below $e_c$). So $fg, e_a, e_b \in C_{e_c} \subseteq \mcO^\uparrow - \mcO$. So $[fg] \fl a \wedge b$ and we have
\[
    [a,f,a][b,f,b] = [a, f e_{a \wedge b} g, b] = [a, fg, b] = [[fg], fg, [fg]] \in \C.
\]
\end{proof}

At this point, by Theorem \ref{thm:labelledgraph}, we have shown that the inverse hull $S$ of a Markov shift is Morita equivalent to a labelled graph inverse semigroup. Moreover, as we will soon see, one can construct the labelled graph associated with $S$ using the following ``combinatorial'' data inside $S$: the set of cores $C_e$ over all $e \in \mcO^\uparrow - \mcO$ together with the idempotents immediately below each core. First we extend the coherent set $\C$ to a larger set $CD(S)$ of idempotents in $S^{\fl}$ that we view as the combinatorial data associated with $S$. 

\begin{defn} Let $S$ be the inverse hull of a Markov shift and let $S^{\fl}$ be defined as above. Define 
\[ 
\C^{\ll} = \{ [b, f, b] : 0 \neq f \ll e_b, \text{and either $f \neq e_{[f]}$ or $[f] \not\fl b$} \}.
\]
We refer to the set $CD(S) = \C \cup \C^{\ll}$ as the \emph{combinatorial data} of $S$.
\end{defn}

The sets $\C$ and $\C^{\ll}$ are disjoint. To see this, suppose $0 \neq [a, e_a, a] = [b, f, b]$ where $f \leq e_b$. Then $a \wedge b \neq 0$ and $f = e_a$. So $f = e_{[f]}$. Moreover, since $[f] \wedge b \neq 0$ we have $e_{[f] \wedge b} = e_{[f]} e_b = e_{[f]}$. Thus $[f] \fl b$. It follows that $[b,f,b] = [[f], e_{[f]}, [f]] \in \C$. Recall that a nonzero idempotent of an inverse semigroup is said to be \emph{primitive} if it is minimal in the set of nonzero idempotents. 

\begin{prop}\label{propCD} The set $\text{CD}(S)$ is a subsemigroup of $S^{\fl}$. Moreover, $\C^{\ll}$ is the set of primitive idempotents of $\text{CD}(S)$.
\end{prop}
\begin{proof} We know $\C$ is closed under multiplication by Proposition \ref{prop:CDineq}. Let $[a,f,a] \in \C^{\ll}$ and $[b, e_b, b] \in \C$. We will show that $[a,f,a] [b, e_b, b] \neq 0$ if and only if $a \fl b$.
First suppose that $[a,f,a] [b, e_b, b] \neq 0$. Then $e_b f \neq 0$ and $a \wedge b \neq 0$. Thus $e_{a \wedge b} = e_a e_b$ and $[b, e_b, b] [a,f,a] = [a \wedge b, e_b f, a \wedge b]$. There are two cases to consider. First suppose that $f \neq e_{[f]}$. In that case $e_b f \in \mcO$ and, since products of distinct elements of $\mcO$ are zero, we must have $f \leq e_b$. Hence $f \leq e_a e_b$. But $f \ll e_a$, so either $f = e_a e_b$ or $e_a = e_a e_b$. The first case is impossible since $f \in \mcO$. In the second case, $e_a \leq e_b$. But $a \wedge b \neq 0$ then implies that $a \fl b$.

Next, suppose $f = e_{[f]}$ and $[f] \not\fl a$. We claim that $f \leq e_a e_b$. If $f = e_{[f]}$ and $e_a e_b$ are incomparable, then since $e_{[f]}e_a e_b = f e_b \neq 0$ it follows that $e_{[f]}, e_a e_b \in C_{e_a}$. Thus $[f] \fl a$, a contradiction. But if $e_a e_b \leq e_{[f]}$, then
\[
    [a \wedge b, e_a e_b, a \wedge b] \leq  [a, f, a] \leq [a, e_a , a]
\]
implies $[a,f,a] \in \C,$ another contradiction. We conclude that $f \leq e_a e_b$. Since $f \ll e_a$ we have $f = e_a e_b$ or $e_a \leq e_b$. We have already seen that $f = e_a e_b$ implies $[a,f,a] \in \C$. So we conclude that $e_a \leq e_b$, and hence $a \fl b$. Conversely, if $a \fl b$ then $f \leq e_a \leq e_b$ and
\[
[a,f,a] [b, e_b, b] = [a, f e_a e_b, b] = [a, f, a]. 
\]
So we have shown that $[a,f,a][b, e_b, b] = [a, f, a]$, when the product is nonzero.

Next suppose that $[a,f,a], [b, g,b] \in C^{\ll}$. We claim that
\[
[a,f,a][b,g,b] = \begin{cases}
                    [a,f,a] & \text{if $a = b$ and $f = g$,}\\
                    0       & \text{otherwise.}
\end{cases}
\]
Of the four cases to consider, we prove the claim when $f = e_{[f]}$, $g = e_{[g]}$, $[f] \not\fl a$, and $[g] \not\fl b$. The other cases involve similar arguments. Suppose that $[a,f,a][b,g,b] \neq 0$. Then $a \wedge b \neq 0$ and $fg \neq 0$. We consider the pair of idempotents $f$ and $e_a e_b$. Suppose that they are incomparable. Since $f e_a e_b \neq 0$, we have $f, e_a e_b \in C_{e_a}$. Thus $[f] \fl a$, a contradiction. So $f$ and $e_a e_b$ are comparable. But if $e_a e_b \leq f \leq e_a$ we have $[f] \fl a$. So $f \leq e_a e_b$. Since $f \ll e_a$, we know $f = e_a e_b$ or $e_a e_b = e_a$. The first case is impossible since it implies that $[f] = a \wedge b \fl a$, so $e_a e_b = e_a$ and hence $a \wedge b = a$. That is, $a \fl b$. A similar argument considering the pair $g$ and $e_a e_b$ shows that $b \fl a$. So $a = b$. Since $f,g \ll e_a$ with nonzero product we quickly see that $f = g$.

The proof that $C^{\ll}$ is the set of primitive idempotents in $CD(S)$ now follows quickly from the rules for multiplication that we have proved. 

% put remaining calculations here
\end{proof}

We now describe the labelled graph associated to the inverse hull of a Markov shift. This is closely related to the combinatorial data of $S$. 

\begin{defn} Let $S$ be the inverse hull of a Markov shift and 
\[\{e_a : a \in S/\D \} =  (\mcO^\uparrow - \mcO) \cup \{0\}\]
be a set of idempotent $\D$-class representatives for $S$. The \emph{labelled graph of $S$} is $(\E, \mcL)$ where
\begin{enumerate}
\item the vertex set $\E^{0} = S / \D - \{0\}$,
\item for $a \in \E^0$, nonzero $f \ll e_a$ and all $b \fl [f]$, there is an edge $x_{a,f,b} \in \E^{1}$ provided that $f \neq e_{[f]}$ or $[f] \not\fl v$, and
\item for $x_{a,f,b} \in \E^{1}, r(x_{a,f,b}) = a, s(x_{a,f,b}) = b$, and $\mcL(x_{a,f,b}) = (a,f)$.
\end{enumerate}
\end{defn}

\begin{exmp} Consider the Markov shift generated by the following transition matrix $T$ with alphabet $\{a,b,c\}$:
    \[T = \kbordermatrix{
    & a & b & c \\
    a & 1 & 1 & 0  \\
    b & 0 & 1 & 1 \\
    c & 1 & 1 & 1 }.\]
The semilattice is:
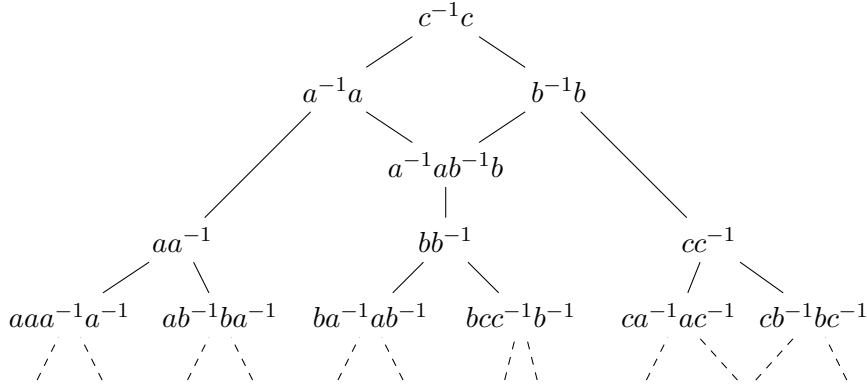
\begin{center}
\begin{tikzpicture}
  % Nodes
  \node (a) at (3.5,2) {$a^{-1}a$};
  \node (b) at (6.5,2) {$b^{-1}b$};
  \node (c) at (5,3) {$c^{-1}c$};
  \node (ab) at (5,1) {$a^{-1}ab^{-1}b$};

  \node (a1) at (1.5,0) {$aa^{-1}$};
  \node (b1) at (5,0) {$bb^{-1}$};
  \node (c1) at (8.5,0) {$cc^{-1}$};

    %delete these node if you just want O and up (and ask and Stian will resize to look nicer)
  \node (aa) at (0,-1) {$aaa^{-1}a^{-1}$};
  \node (ab1) at (2,-1) {$ab^{-1}ba^{-1}$};
  \node (ba) at (4,-1) {$ba^{-1}ab^{-1}$};
  \node (bc) at (6,-1) {$bcc^{-1}b^{-1}$};
  \node (ca) at (8.1,-1) {$ca^{-1}ac^{-1}$};
  \node (cb) at (9.9,-1) {$cb^{-1}bc^{-1}$};

  \node (aa1) at (-.5,-2) {};
  \node (ab11) at (1.5,-2) {};
  \node (aa2) at (.5,-2) {};
  \node (ab12) at (2.5,-2) {};
  \node (ba1) at (3.5,-2) {};
  \node (ba2) at (4.5,-2) {};
  \node (bc1) at (5.75,-2) {};
  \node (bc2) at (6.25,-2) {};

  \node (ca1) at (7.6,-2) {};
  \node (ca2) at (9,-2) {};
  \node (cb2) at (10.4,-2) {};

  % Lines
  \draw (a) -- (c);
  \draw (b) -- (c);
  \draw (a) -- (ab);
  \draw (b) -- (ab);
  \draw (a1) -- (a);
  \draw (b1) -- (ab);
  \draw (c1) -- (b);

    % delete these lines if you just want O and up 
  \draw (a1) -- (aa);
  \draw (a1) -- (ab1);
  \draw (b1) -- (ba);
  \draw (b1) -- (bc);
  \draw (c1) -- (ca);
  \draw (c1) -- (cb);

  \draw[dashed] (aa1) -- (aa);
  \draw[dashed] (ab11) -- (ab1);
  \draw[dashed] (aa2) -- (aa);
  \draw[dashed] (ab12) -- (ab1);
  \draw[dashed] (ba1) -- (ba);
  \draw[dashed] (ba2) -- (ba);
  \draw[dashed] (bc1) -- (bc);
  \draw[dashed] (bc2) -- (bc);
  \draw[dashed] (ca1) -- (ca);
  \draw[dashed] (ca2) -- (ca);
  \draw[dashed] (ca2) -- (cb);
  \draw[dashed] (cb2) -- (cb);
\end{tikzpicture}
\end{center}
Here we have $\mcO = \{aa^{-1}, bb^{-1}, cc^{-1}\}$ and there are four nonzero $\D$-classes with representatives $\{e_c = c^{-1} c, e_a = a^{-1} a, e_b = b^{-1} b, e_d = a^{-1} a b^{-1} b\}$. The vertex set is $\E^0 = \{a, b, c, d\}$ with $d \fl a, d \fl b, a\fl c$, $b \fl c$, and $d \fl c$. We have labels $\mcL = \{ \alpha, \beta, \gamma\}$ where $\alpha = (a, aa^{-1}), \beta = (d, bb^{-1})$ and $\gamma = (b, cc^{-1})$. The labelled graph is pictured below.
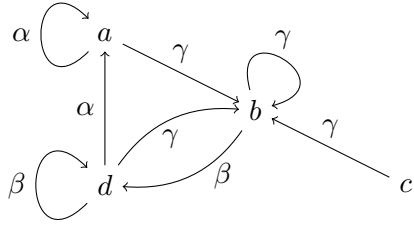
\begin{center}
    \begin{tikzpicture}
        \node (a) at (0, 1) {$a$};
        \node (b) at (2, 0) {$b$};
        \node (c) at (4,-1) {$c$};
        \node (d) at (0, -1) {$d$};

        \draw[->] (a) to[loop, out=225, in=135, looseness=7] node[left] {$\alpha$} (a);
        \draw[->] (d) to node[left] {$\alpha$} (a);
        
        \draw[->] (b) to[bend left] node[right=5] {$\beta$} (d);
        \draw[->] (d) to[loop, out=225, in=135, looseness=7] node[left] {$\beta$} (d);

        \draw[->] (a) to node[above] {$\gamma$} (b);
        \draw[->] (b) to[loop, out=105, in=15, looseness=8] node[above] {$\gamma$} (b);
        \draw[->] (c) to node[above] {$\gamma$} (b);
        \draw[->] (d) to[bend left] node[below] {$\gamma$} (b);
    \end{tikzpicture}
\end{center}
\end{exmp}

\begin{defn} An \emph{isomorphism of labelled graphs} $(\E_1, \mcL_1)$ and $(\E_2, \mcL_2)$ is a triple $\pi = (\pi_0, \pi_1, \pi_2)$ consisting of bijections $\pi_0 : \E_1^0 \to \E_2^0$, $\pi_1 : \E_1^1 \to \E_2^1$, $\pi_2 : \A_1 \to \A_2$ such that for each edge $e \in \E_1^1$, $s(\pi_1(e)) = \pi_0(s(e))$, $r(\pi_1(e)) = \pi_0(r(e))$, and $\mcL(\pi_1(e)) = \pi_2( \mcL(e))$.
\end{defn}

\begin{prop}\label{prop:graphiso} Suppose that $S$ is the inverse hull of a Markov shift. Then the labelled graph associated with $S$ is isomorphic to the labelled graph associated with the coherent set $\C$ in $S^{\fl}$.
\end{prop}
\begin{proof} 

Consider the map $\sigma : S / \D  \to S^{\fl} / \D$ defined by $\sigma([f]) = [f']$, where $f' = [[f],f,[f]]$. By Proposition \ref{prop:flD}, $\sigma$ is a well-defined bijection and the restriction to nonzero $\D$-classes defines a bijection $\sigma : \E^0_1 \to \E^0_2$. Let $x_{a,f,b}$ be an edge in the labelled graph associated with $S$. Then $f \ll e_a$ and $b \fl [f]$. Also, either $f \neq e_{[f]}$ or $[f] \not\fl a$. Let $\overline{f} = [a, f, a]$. We claim that $x_{\sigma(a), \overline{f}, \sigma(b)}$ is an edge in the labelled graph associated with $S^{\fl}$. First note that $0 \neq [a,f,a] \ll [a, e_a, a]$. Also $e_b \leq e_{[f]}$ by Lemma \ref{lem:ineq}, so 
\[[
b,e_b,b] = [[f],e_b,[f]] \leq [[f],e_{[f]},[f]].
\]
Thus, $\sigma(b) \in B_{[\overline{f}]}$. Finally, suppose for the sake of contradiction that $\overline{f} = e_{[\overline{f}]}$. Then $[a, f, a] = [[f], e_{f}, [f]]$. So $f = e_{[f]}$ and $a \wedge [f] \neq 0$. Then $e_{[f]} = e_{[f]} e_a = e_{[f] \wedge a}$. Thus $[f] \fl a$. So we have that $f = e_{[f]}$ and $[f] \fl a$, a contradiction. This shows that $\overline{f} \neq e_{[\overline{f}]}$, and we conclude $x_{\sigma(a), \overline{f}, \sigma(b)}$ is an edge. It is easily verified that $\pi_1 : \E^1_1 \to \E^1_2$ defined by
\[
\pi_1(x_{a,f,b}) = x_{\sigma(a), \overline{f}, \sigma(b)} 
\]
is a bijection. 
By the same reasoning, the map $\pi_2$ sending $(a,f) \mapsto (\sigma(a), \overline{f})$ defines a bijection on labels. One can quickly check that the maps satisfy $s(\pi_1(e)) = \sigma(s(e))$, $r(\pi_1(e)) = \sigma(r(e))$, and $\mcL(\pi_1(e)) = \pi_2(\mcL(e))$.

\end{proof}

Finally we prove the main theorem of the paper which shows that the labelled graph of the inverse hull of a Markov shift is a complete Morita equivalence invariant among inverse hulls of Markov shifts.

\begin{thm}\label{thm:markovmorita} Suppose that $S$ and $T$ are the inverse hulls of Markov shifts. Then the following are equivalent:
\begin{enumerate}
\item $S$ is Morita equivalent to $T$.
\item There exists a $\D$-class preserving isomorphism 
\[\pi : CD(S^{\fl}) \to CD(T^{\fl}).\]
\item The labelled graphs associated with $S$ and $T$ are isomorphic with the map on vertices being an order isomorphism.
\end{enumerate}
\end{thm}
\begin{proof}
Throughout the proof assume that we have chosen sets 
\[
\{e_a : a \in S/\D \} =  (\mcO_{S}^\uparrow - \mcO_{S}) \cup \{0\} \text{ and } \{f_v : v \in T/\D \} =  (\mcO_{T}^\uparrow - \mcO_{T}) \cup \{0\}
\]
of $\D$-class representatives of $S$ and $T$ respectively.

(1) $\Rightarrow$ (2): First, suppose that $S$ is Morita equivalent to $T$ via an equivalence functor $\F:C(S) \to C(T)$. We will prove that there is a $\D$-class preserving isomorphism $\pi : CD(S^{\fl}) \to CD(T^{\fl})$. By Proposition \ref{lem:equivalence} (1), there exists a bijection $\sigma : S/\D \to T/\D$, such that $\sigma([e]) = [\F(e)]$ for all idempotents $e \in S$. We define $\pi$ on $\C_{S}$ by $\pi( [a, e_a, a] ) = [\sigma(a), f_{\sigma(a)}, \sigma(a)]$. Next, for a given nonzero $a \in S/\D$, let $e_a^{\ll} = \{f : 0\neq f \ll e_a\}$. Then $e_a^{\ll}$ is the disjoint union of the sets $e_{a,b}^{\ll} = \{f \in e_a^{\ll} : [f] = b\}$ over all nonzero $b \in S / \D$. By Proposition \ref{lem:equivalence} (2), there is a $\D$-preserving isomorphism $\tau_a : e_a S e_a \to f_{\sigma(a)} T f_{\sigma(a)}$. 

For nonzero $b \in S/ \D$ with $b \not\fl a$, the restriction of $\tau_{a}$ to $e_{a,b}^{\ll}$ gives a bijection $\pi_{a,b} : e_{a,b}^{\ll} \to f_{\sigma(a),\sigma(b)}^{\ll}$. Next suppose $b \fl a$. We claim that $e_b \ll e_a$ if and only if $f_{\sigma(b)} \ll f_{\sigma(a)}$. Suppose $e_b \ll e_a$. Then $e_b \in C_{e_a}$ and there exists $f \in \tau_a(C_{e_a}) = C_{f_{\sigma(a)}}$ with $f \ll f_{\sigma(a)}$ and $f \D f_{\sigma(b)}$. As noted in the proof of Lemma \ref{lem:ineq}, $C_{f_{\sigma(a)}} \subseteq \mcO^{\uparrow}_T - \mcO_T$. Thus $f = f_{\sigma(b)}$. The converse is similar. It follows that we can find a bijection $\pi_{a,b} : e_{a,b}^{\ll} - \{e_b\} \to f_{\sigma(a),\sigma(b)}^{\ll} - \{f_{\sigma(b)}\}$, possibly different than the restriction of $\tau_{a}$. For $[a,f,a] \in \C^{\ll}_{S^{\fl}}$ we define $\pi([a,f,a]) = [\sigma(a), \pi_{a,[f]}(f), \sigma(a)]$. One can quickly verify that $\pi$ is well-defined. (Note the significant detail in the case that $[f] \fl a$ that the range of $\pi_{a,[f]}$ does not include $f_{\sigma{([f])}}$.)

By construction, $\pi$ is a $\D$-class preserving bijection. We show that $\pi$ is an order isomorphism (and hence a semilattice isomorphism as well). Let $E,F \in CD(S^{\fl})$ with $E \leq F$. We have already considered the case that $E,F \in \C_{S^{\fl}}$. We may assume $0 \neq E \in \C^{\ll}_{S^{\fl}}$. If $F \in \C^{\ll}_{S^{\fl}}$, then by Proposition \ref{propCD}, $F = E$ and $\pi(E) = \pi(F)$. Suppose $F \in \C_{S^{\fl}}$ and write $E = [a, f, a], F = [b, e_b, b]$. From the proof of Proposition \ref{propCD} we know that $a \fl b$. Thus $f \ll e_a \leq e_b$. We have that $\pi_{a,[f]}(f) \leq f_{\sigma(a)}$. Moreover $a \fl b$ implies $f_{\sigma(a)} \leq f_{\sigma(b)}$. Thus $\pi_{a,[f]}(f) \leq f_{\sigma(a)} \leq f_{\sigma(b)}$ and hence $\pi(F) \leq \pi(E)$. By a similar argument, if $\pi(E) \leq \pi(F)$ then $E \leq F$. Thus $\pi$ is a $\D$-class preserving isomorphism.

(2) $\Rightarrow$ (3): Next suppose that $\pi : CD(S^{\fl}) \to CD(T^{\fl})$ is a $\D$-class preserving isomorphism and let $(\E_1, \mcL_1)$ and $(\E_2, \mcL_2)$ be the labelled graphs with label sets $\A_1$ and $\A_2$ associated with $S$ and $T$ respectively. Define a map $\sigma_{0} : \E^{0}_1 \to \E^{0}_2$ by $\sigma(a) = u$ if and only if $\pi([a, e_a, a]) = [u, e_u, u]$. Note that $[v, e_v, v] = [u, e_u, u]$ if and only if $u = v$, so $\sigma_0$ is a well-defined bijection. We also have $a \fl b$ if and only if $\sigma_0(a) \fl \sigma_0(b)$, so that $\sigma_0$ is an order isomorphism.

Next, let $[a,f,a] \in \C_S^{\ll}$ and write $\pi([a,f,a]) = [u, g, u]$, where $[u,g,u] \ll [u, f_u, u]$. We aim to show that $u = \sigma_0(a)$. We have that $[u,g,u] \ll [u, f_u, u]$ and $[u,g,u] \ll \pi([a,e_a,a]) = [\sigma_0(a), f_{\sigma_0(a)}, \sigma_0(a)]$. Thus $u \wedge \sigma_0(a) \neq 0, g \ll f_u,$ and $g \ll f_{\sigma_0(a)}$. If $f_u$ and $f_{\sigma_0(a)}$ are incomparable, then it follows that $g = e_{[g]}$ with $[g] \fl u$, a contradiction. Then since $g \ll f_u,$ and $g \ll f_{\sigma_0(a)}$ we must have $f_u = f_{\sigma_0(a)}$ and $u = \sigma_0(a)$.

Suppose now that $x_{a,f,b} \in \E_1^1$. We define a map $\sigma_1 : \E_1^1 \to \E_2^1$ by $\sigma_1( x_{a,f,b} ) = x_{\sigma_0(a), g, \sigma_0(b)}$ where $\pi([a,f,a]) = [\sigma_0(a), g, \sigma_0(a)]$. We verify that $\sigma_1( x_{a,f,b} )$ is indeed an edge in $\E_2$. Then it can be quickly shown that $\sigma_1$ is a well-defined bijection. Since $x_{a,f,b} \in \E_1^1$, we have $f \ll e_a$, $b \fl [f]$, and either $f \neq e_{[f]}$ or $[f] \not\fl a$. Now $g \ll f_{\sigma_0(a)}$ with either $g \neq f_{[g]}$ or $[g] \not\fl \sigma_0(a)$. Also,
\begin{align*}
[\sigma_0(b), e_{\sigma_0(b)}, \sigma_0(b)] &= \pi([b,e_b,b]) \\
                                            &\leq \pi([[f],e_{[f]},[f]]) \\
                                            &= [\sigma_0([f]), e_{\sigma_0([f])}, \sigma_0([f])]
\end{align*}
implies that $\sigma_0(b) \leq \sigma_0([f])$. Thus $\sigma_1 : \E_1^1 \to \E_2^1$ is a bijection. Similarly, $(a,f) \mapsto (\sigma_0(a), g)$ defines a bijection $\sigma_2$ from $\A_1$ to $\A_2$. For each edge $e$ we have $s(\sigma_1(e)) = \sigma_0(s(e))$, $r(\sigma_1(e)) = \sigma_0(r(e))$, and $\mcL(\sigma_1(e)) = \sigma_2( \mcL(e))$. Thus $(\E_1, \mcL_1)$ is isomorphic to $(\E_2, \mcL_2)$. 

(3) $\Rightarrow$ (1): Finally, let $(\E_1, \mcL_1)$ and $(\E_2, \mcL_2)$ be the labelled graphs associated with $S$ and $T$ respectively and suppose that $\pi = (\pi_0, \pi_1, \pi_2)$ is an isomorphism from $(\E_1, \mcL_1)$ to $(\E_2, \mcL_2)$ and $\pi_0$ is an order isomorphism. By Proposition \ref{prop:graphiso}, the labelled graphs associated with the coherent sets $\C_S$ and $\C_T$ are isomorphic via a map $\pi' = (\pi'_0, \pi'_1, \pi'_2)$ where $\pi'_0([a, e_a, a]) = [\pi_0(a), e_{\pi_0(a)},\pi_0(a)]$ is an order isomorphism. It follows that the labelled graph spaces associated with $\C_S$ and $\C_T$ are isomorphic and hence the corresponding labelled graph inverse semigroups are isomorphic. Thus $S^{\fl}$ is Morita equivalent to $T^{\fl}$ and hence $S$ is Morita equivalent to $T$.
% By Lemma \ref{lem:equivalence}, we have a bijection $\sigma : S/\D \to T/\D$ where $\sigma([e]) = [\F(e)]$ for each idempotent $e$ in $S$. We note that $\sigma$ is in fact an order isomorphism from $(S/\D, \fl_S)$ to $(T/\D, \fl_T)$ because $\F$ induces an isomorphism on each local monoid of $S$ and for each $e \in E(S)$, the core $C_e \subseteq eSe$. It follows immediately that $\sigma$ induces a bijection from $\B_S$ to $\B_T$ via $B_a \mapsto B_{\sigma(a)}$ for each $a \in S/\D$.

% Let $\{e_a : a \in S / \D\} = (\mcO_S^\uparrow - \mcO_S) \cup \{0\}$ be a set of $\D$-class representatives for $S$. Then we may choose a set $\{f_{\sigma(a)} : a \in S / \D\} = (\mcO_T^\uparrow - \mcO_T) \cup \{0\}$ of $\D$-class representatives for $T$ where $[f_{\sigma(a)}] = [\sigma(a)] = [\F(e_a)]$. By Lemma \ref{lem:equivalence} there exists an isomorphism
% \[
%     \tau_a : e_a S e_a \to f_{\sigma(a)} T f_{\sigma(a)}.
% \]
% It follows that the map $x_{a,f,b} \mapsto x_{\sigma(a), \tau_a(f), \sigma(b)}$ is a bijection from $\E_S^1 \to \E_T^1$ that is compatible with sources, ranges, and labels. Thus the labelled graphs of $S$ and $T$ are isomorphic. 
\end{proof}
\bibliographystyle{amsplain}
\bibliography{SemigroupBib.bib}
\end{document}